%% file: ordinal-embedding-v10-post-BEJ.tex
\documentclass[twoside,11pt]{article}

\usepackage{graphicx, subfigure}
\usepackage[top=1in,bottom=1in,left=1in,right=1in]{geometry}
\usepackage[sort&compress]{natbib} 
\usepackage{amsfonts, amsmath, amssymb, amsthm, constants, MnSymbol}

\usepackage{color}
\definecolor{darkred}{RGB}{100,0,0}
\definecolor{darkgreen}{RGB}{0,100,0}
\definecolor{darkblue}{RGB}{0,0,150}

\usepackage{hyperref}
\hypersetup{colorlinks=true, linkcolor=darkred, citecolor=darkgreen, urlcolor=darkblue}
\usepackage{url}

\input notation.tex

\definecolor{purple}{rgb}{0.4,.1,.9}

\begin{document}
\thispagestyle{empty}

\title{Some Theory for Ordinal Embedding}
\author{
Ery Arias-Castro\footnote{Department of Mathematics, University of California, San Diego, USA}%
}
\date{}
\maketitle

\begin{abstract}\noindent
Motivated by recent work on ordinal embedding \citep{klein}, we derive large sample consistency results and rates of convergence for the problem of embedding points based on triple or quadruple distance comparisons.  We also consider a variant of this problem where only local comparisons are provided.  Finally, inspired by \citep{jamieson2011low}, we bound the number of such comparisons needed to achieve consistency.

\end{abstract}

\section{Introduction} \label{sec:intro}

The problem of ordinal embedding, also called non-metric multidimensional scaling \citep{borg2005modern}, consists of finding an embedding of a set of items based on pairwise distance comparisons.  Specifically, suppose that $\delta_{ij} \ge 0$ is some dissimilarity measure between items $i, j \in [n] := \{1, \dots, n\}$.  We assume that $\delta_{ii} = 0$ and $\delta_{ij} = \delta_{ji}$ for all $i, j \in [n]$. 
These dissimilarities are either directly available but assumed to lack meaning except for their relative magnitudes, or only available via comparisons with some other dissimilarities, meaning that we are only provided with a subset $\cC \subset [n]^4$ such that   
\beq \label{C}
\delta_{ij} < \delta_{k\ell}, \quad \forall (i,j,k,\ell) \in \cC.
\eeq
Note that the latter setting encompasses the former.
Given $\cC$ and a dimension $d$, the goal is to embed the items as points $p_1, \dots, p_n \in \bbR^d$ in a way that is compatible with the available information, specifically
\beq \label{embed}
\delta_{ij} < \delta_{k\ell} \implies \|p_i - p_j\| \le \|p_k - p_\ell\|, \quad \forall (i,j,k,\ell) \in \cC,
\eeq
where $\|\cdot\|$ denotes the Euclidean norm.
The two most common situations are when all the quadruple comparisons are available, meaning $\cC = [n]^4$, or all triple comparisons are available, meaning $\cC = \{(i,j,i,k): i,j,k \in [n]\}$, which can be identified with $[n]^3$.
This problem has a long history surveyed in \citep{young1987multidimensional}, with pioneering contributions from  \cite{MR0140376,MR0173342} and \cite{MR0169712}. 

The main question we tackle here is that  of {\em consistency}.  Suppose that the items are in fact points $x_1, \dots, x_n \in \bbR^d$ and $\delta_{ij} = \|x_i - x_j\|$.  (When the $\delta_{ij}$'s are available, suppose that $\delta_{ij} = g(\|x_i - x_j\|)$ where $g$ is an unknown increasing function.)  Provided with a subset $\cC = \cC_n$ of dissimilarity comparisons as in \eqref{embed}, is it possible to reconstruct the original points in the large-sample limit $n \to \infty$?
Clearly, the reconstruction can only be up to a similarity transformation
--- that is, a transformation $f : \bbR^d \mapsto \bbR^d$ such that, for some $\lambda > 0$, $\|f(x) - f(y)\| = \lambda \|x - y\|$ for all $x,y \in \bbR^d$, or equivalently, of the form $f(x) = \lambda R(x) + b$ where $R$ is an orthogonal transformation and $b$ is a constant vector --- since such a transformation leaves the distance comparisons unchanged.  
This question is at the foundation of non-metric multidimensional scaling.  

Early work only addressed the continuous case, where the $x$'s span a whole convex subset $U \subset \bbR^d$.  In that setting, the goal becomes to characterize {\em isotonic functions} on $U$, that is, functions $f : U \mapsto \bbR^d$ satisfying
\beq \label{isotonic}
\|x - y\| < \|x' - y'\| \implies \|f(x) - f(y)\| \le \|f(x') - f(y')\|, \quad \forall x,y,x',y' \in U.
\eeq 
\cite{shepard1966metric} argues that such functions must be similarities, and cites earlier work \citep{suppes1955axiomatization,aumann1958coefficients} dealing with the case $d=1$.

Only recently has the finite sample case been formally considered. 
Indeed, \cite{klein} prove a consistency result, showing that if $x_1, \dots, x_n \in U \subset \bbR^d$, where $U$ is a bounded, connected, and open subset of $\bbR^d$ satisfying some additional conditions --- for example, a finite union of open balls --- and $\cC = [n]^4$, then in the large sample limit with $x_1, \dots, x_n$ becoming dense in $U$, it is possible to recover the $x$'s up to a similarity transformation.  
(Note that $U$ is then uniquely defined as the interior of $\overline{\{x_i : i \ge 1\}}$.)
We note that \cite{klein} focus on the strictly isotonic case, where the second inequality in \eqref{isotonic} is strict.
Our first contribution is an extension of this consistency result for quadruple learning to triple learning where $\cC = [n]^3$.
In the process, we greatly simplify the arguments of \cite{klein} and weaken the conditions on the sampling domain $U$.
We note that \cite{terada2014local} have partially solved this problem by a reduction to the problem of embedding a nearest-neighbor graph.  However, their arguments are based on an apparently incomplete proof in \citep{von2013density}, which is itself based on a rather sophisticated approach.  Our proofs are comparatively much simpler and direct.
 
Our second contribution is to provide rates of convergence, a problem left open by \cite{klein}.  
In the context of quadruple learning, we obtain a rate in $O(\eps_n)$, where $\eps_n$ is the Hausdorff distance between the underlying sample $\{x_1, \dots, x_n\}$ and $U$, meaning, $\eps_n := \sup_{x \in U} \min_{i \in [n]} \|x - x_i\|$.  This is the first convergence rate for exact ordinal embedding that we know of.  (We are not able to obtain the same rate in the context of triple learning.)  
Compared to establishing consistency, the proof is much more involved.   

The last decade has seen a surge of interest in ordinal embedding, motivated by applications to recommender systems and large-scale psychometric studies made available via the internet, for example, databases for music artists similarity \citep{ellis2002quest,mcfee2011learning}.  Sensor localization \citep{nhat2008nonmetric} is another possible application.
Modern datasets being large, all quadruple or triple comparisons are rarely available, motivating the proposal of embedding methods based on a sparse set of comparisons \citep{borg2005modern,agarwal2007generalized,jamieson2011low,terada2014local}.  
\cite{terada2014local} study what they call {\em local ordinal embedding}, which they define as the problem of embedding an unweighted $K$-nearest neighbor ($K$-NN) graph.  With our notation, this is the situation where $\cC = \{(i,j,k): \delta_{ij} \le \delta_{i(K)} < \delta_{ik}\}$, $\delta_{i(K)}$ being the dissimilarity between item $i$ and its $K$th nearest-neighbor.  
\cite{terada2014local} argue that, when the items are points $x_1, \dots, x_n$ sampled from a smooth density on a bounded, connected, convex, and open subset $U \subset \bbR^d$ with smooth boundary, then $K = K_n \gg n^{2/(2+d)} (\log n)^{d/(2+d)}$ is enough for consistency.  
Our third contribution is to consider the related situation where $\cC = \{(i,j,k,\ell): \delta_{ij} < \delta_{k\ell} \text{ and } \max(\delta_{ij}, \delta_{ik}, \delta_{i\ell}) \le \delta_{i(K)}\}$, which provides us with the $K$-NN graph and also all the quadruple comparisons between the nearest neighbors.  In this setting, we are only able to show that $K_n \gg \sqrt{n \log n}$ is enough.      

Beyond local designs, which may not be feasible in some settings, \cite{jamieson2011low} consider the problem of adaptively (i.e., sequentially) selecting triple comparisons in order to minimize the number of such comparisons and yet deduce all the other triple comparisons. 
They consider a few methods, among which a non-metric version of the  landmark MDS method of \cite{de2004sparse}.
Less ambitious is the problem of selecting few comparisons in order to consistently embed the items when these are points in a Euclidean space.   
Our fourth contribution is to show that one can obtain a consistent embedding with a landmark design based on $a_n n$ queries, where $a_n$ is any diverging sequence.  Moreover, the embedding can be computed in (expected) time $\zeta(a_n) \, n$, for some function $\zeta : \bbR_+ \mapsto \bbR_+$. 

The rest of the paper is organized as follows.  
In \secref{theory}, we state our theoretical results and prove the simpler ones.
We then gather the remaining proofs in \secref{proofs}.
\secref{discussion} concludes the paper with a short discussion.

\section{Theory} \label{sec:theory}

In this section we present our theoretical findings.  Most proofs are gathered in \secref{proofs}.

We already defined isotonic functions in \eqref{isotonic}.
Following \citep{klein}, we say that a function $f : U \subset \bbR^d \mapsto \bbR^d$ is {\em weakly isotonic} if 
\beq \label{weak-isotonic}
\|x - y\| < \|x - z\| \implies \|f(x) - f(y)\| \le \|f(x) - f(z)\|, \quad \forall x,y,z \in U.
\eeq 
Obviously, if a function is isotonic \eqref{isotonic}, then it is weakly isotonic \eqref{weak-isotonic}. 
Weak isotonicity is in fact not much weaker than isotonicity.
Indeed, let $P$ be a property (e.g., `isotonic'), and say that a function $f : U \subset \bbR^d \mapsto \bbR^d$ has the property $P$ locally if for each $x \in U$ there is $r > 0$ such that $f$ has property $P$ on $B(x, r) \cap U$, where $B(x,r)$ denotes the open ball with center $x$ and radius $r$.

\begin{lem} \label{lem:local-weak}
Any locally weakly isotonic function on an open $U$ is also locally isotonic on $U$.  
\end{lem}

\begin{proof}
This is an immediate consequence of \citep[Lem 6]{klein}, which implies that a weakly isotonic function on $B(x, r)$ is isotonic on $B(x, r/4)$.
\end{proof}

Suppose we have data points $x_1, \dots, x_n \in \bbR^d$.
Define
\beq \label{omega}
\Omega_n = \{x_1, \dots, x_n\}, \quad \Omega = \bigcup_{n \ge 1} \Omega_n = \{x_n : n \ge 1\}.
\eeq
Let $\delta_{ij} = \|x_i - x_j\|$ and suppose that we are only provided with a subset $\cC_n\subset [n]^4$ of distance comparisons as in \eqref{C}.  
To an (exact) ordinal embedding $p: [n] \mapsto \bbR^d$ --- which by definition satisfies \eqref{embed} --- we associate the map $\phi_n : \Omega_n \mapsto \bbR^d$ defined by $\phi_n(x_i) = p_i$ for all $i \in [n]$.
We crucially observe that, in the case of all quadruple comparisons ($\cC_n = [n]^4$), the resulting map $\phi_n$ is isotonic on $\Omega_n$; in the case of all triple comparisons ($\cC_n = [n]^3$), $\phi_n$ is only weakly isotonic on $\Omega_n$, instead.
In light of this, and the fact that the location, orientation and scale are all lost when only ordinal information is available, the problem of proving consistency of (exact) ordinal embedding reduces to showing that any such embedding is close to a similarity transformation as the sample size increases, $n \to \infty$.  
This is exactly what \cite{klein} do under some assumptions.

\subsection{Ordinal embedding based on all triple comparisons}

Our first contribution is to extend the consistency results of \cite{klein} on quadruple learning to triple learning.  Following their presentation, we start with a result where the sample is infinite, which is only a mild generalization of \citep[Th 3]{klein}.  

\begin{thm} \label{thm:infinite}
Let $U \subset \bbR^d$ be bounded, connected and open.  Suppose $\Omega$ is dense in $U$ and consider a locally weakly isotonic function $\phi : \Omega \mapsto \bbR^d$.  Then there is a similarity transformation $S$ that coincides with $\phi$ on $\Omega$.  
\end{thm}

The proof is largely based on that of \citep[Th 3]{klein}, but a bit simpler; see \secref{proof-thm-infinite}.  

We remark that there can only be one similarity with the above property, since similarities are affine transformations, and two affine transformations of $\bbR^d$ that coincide on $d+1$ affine independent points are necessarily identical.

In this theorem, the set $\Omega$ is dense in an open subset of $\bbR^d$, and therefore infinite.  In fact, \cite{klein} use this theorem as an intermediary result for proving consistency as the sample size increases.
Most of their paper is dedicated to establishing this, as their arguments are quite elaborate.
We found a more direct route by `tending to the limit as soon as possible', based on \lemref{diagonal} below, which is at the core of the Arzel\`a-Ascoli theorem.

For the remaining of this section, we consider the finite sample setting:
\beq \label{setting}
\begin{array}{c}
\text{ $U \subset \bbR^d$ is bounded, connected and open,} \\
\text{$\Omega_n = \{x_1, \dots, x_n\} \subset U$ is such that $\Omega := \{x_n : n \ge 1\}$ is dense in $U$,} \\
\text{and $\phi_n : \Omega_n \mapsto Q \subset \bbR^d$ is a function with values in a bounded set $Q$.}
\end{array}
\eeq

In the context of \eqref{setting}, we implicitly extend $\phi_n$ to $\Omega$, for example, by setting $\phi_n(x) = q$ for all $x \in \Omega \setminus \Omega_n$, where $q$ is a given point in $Q$, although the following holds for any extension.

\begin{lem} \label{lem:diagonal}
Consider $\Omega_n \subset \bbR^d$ finite and $\phi_n : \Omega_n \mapsto Q \subset \bbR^d$, where $Q$ is bounded.  Then there is $N \subset \bbN$ infinite such that $\phi(x) := \lim_{n \in N} \phi_n(x)$ exists for all $x \in \Omega := \bigcup_n \Omega_n$.
\end{lem}  

This is called the diagonal process in \citep[Problem D, Ch 7]{kelley1975general}.  Although the result is classical, we provide a proof for completeness. 

\begin{proof}
Without loss of generality, suppose $\Omega_n = \{x_1, \dots, x_n\}$.
Let $N_0 = \bbN$.
Since $(\phi_n(x_1) : n \in N_0) \in Q$ and $Q$ is bounded, there is $N_1 \subset N_0$ infinite such that $\lim_{n \in N_1} \phi_n(x_1)$ exists.  In turn, since $(\phi_n(x_2) : n \in N_1)$ is bounded, there is $N_2 \subset N_1$ infinite such that $\lim_{n \in N_2} \phi_n(x_2)$ exists.  Continuing this process --- which formally corresponds to a recursion --- we obtain $\cdots \subset N_{k+1} \subset N_k \subset \cdots \subset N_1 \subset N_0 = \bbN$ such that, for all $k$, $N_k$ is infinite and $\lim_{n \in N_k} \phi_n(x_k)$ exists.  Let $n_k$ denote the $k$th element (in increasing order) of $N_k$ and note that $(n_k : k \ge 1)$ is strictly increasing.  Define $N = \{n_k : k \ge 1\}$.  Since $\{n_p, p \ge k\}  \subset N_k$, we have $\lim_{n \in N} \phi_n(x_k) = \lim_{n \in N_k} \phi_n (x_k)$, and this is valid for all $k \ge 1$. 
\end{proof}

\begin{cor} \label{cor:C1}
Consider the setting \eqref{setting} and assume that $\phi_n$ is weakly isotonic.  Then $(\phi_n)$ is sequentially pre-compact for the pointwise convergence topology for functions on $\Omega$ and all the functions where it accumulates are similarity transformations restricted to $\Omega$.  
\end{cor}

The corresponding result \citep[Th 4]{klein} was obtained for isotonic (instead of weakly isotonic) functions and for domains $U$ that are finite unions of balls, and the convergence was uniform instead of pointwise.  For now, we provide a proof of Corollary~\ref{cor:C1}, which we derive as a simple consequence of \thmref{infinite} and \lemref{diagonal}.  

\begin{proof}
\lemref{diagonal} implies that $(\phi_n)$ is sequentially pre-compact for the pointwise convergence topology.
Let $\phi$ be an accumulation point of $(\phi_n)$, meaning that there is $N \subset \bbN$ infinite such that $\phi(x) = \lim_{n \in N} \phi_n(x)$ for all $x \in \Omega$.  Take $x,y,z \in \Omega$ such that $\|x-y\| < \|x-z\|$.  By definition, there is $m$ such that $x,y,z \in\Omega_m$, and therefore $\|\phi_n(x)- \phi_n(y)\| \le \|\phi_n(x)- \phi_n(z)\|$ for all $n \ge m$.  Passing to the limit along $n \in N$, we obtain $\|\phi(x)- \phi(y)\| \le \|\phi(x)- \phi(z)\|$.
Hence, $\phi$ is weakly isotonic on $\Omega$ and, by \thmref{infinite}, it is therefore the restriction of a similarity transformation to $\Omega$.  
\end{proof}

It is true that \citep[Th 4]{klein} establishes a uniform convergence result.  We do the same in \thmref{consistent} below, but with much simpler arguments.  The key are the following two results bounding the modulus of continuity of a (resp.~weakly) isotonic function.  We note that the second result (for weakly isotonic functions) is very weak but sufficient for our purposes here.  
%
For $\Lambda \subset V \subset \bbR^d$, define $\delta_H(\Lambda, V) = \sup_{y \in V} \inf_{x \in \Lambda} \|y - x\|$, which is their Hausdorff distance.  
We say that $(y_i : i \in I) \subset \bbR^d$ is an $\eta$-packing if $\|y_i - y_j\| \ge \eta$ for all $i \ne j$.
We recall that the size of the largest $\eta$-packing of a Euclidean ball of radius $r$ is of exact order $(r/\eta)^{-d}$.
For a set $V \subset \bbR^d$, let $\diam(V) = \sup_{x,y \in V} \|x -y\|$ be its diameter and let 
\beq \label{rho}
\rho(V) = \arg \sup_{r > 0} \{\exists v \in V : B(v,r/2) \subset V\},
\eeq
which is the diameter of a largest ball inscribed in $V$.

Everywhere in the paper, $d$ is fixed, and in fact implicitly small as we assume repeatedly that the sample (of size $n$) is dense in a full-dimensional domain of $\bbR^d$.  
In particular, all the implicit constants of proportionality that follow depend solely on $d$.

\begin{lem} \label{lem:Lip}
Let $V \subset \bbR^d$ be open.  Consider $\Lambda \subset V$ and set $\eps = \delta_H(\Lambda, V)$.  Let $\psi : \Lambda \mapsto Q$ be isotonic, where $Q \subset \bbR^d$ is bounded.  There is $C \propto \diam(Q)/\rho(V)$, such that
\beq \label{Lip}
\|\psi(x) - \psi(x')\| \le C (\|x - x'\| + \eps), \quad \forall x,x' \in \Lambda.
\eeq
\end{lem}

\begin{proof}
The proof is based on the fact that an isotonic function transforms a packing into a packing.  
Take $x,x' \in \Lambda$ such that $\xi := \|\psi(x) - \psi(x')\| > 0$, and let $\eta = \|x - x'\|$.
Since $V$ is open it contains an open ball of diameter $\rho(V)$.  Let $y_1, \dots, y_m$ be an $(\eta + 3 \eps)$-packing of such a ball with $m \ge C_1 (\rho(V)/(\eta + \eps))^{d}$ for some constant $C_1$ depending only on $d$.  Then let $x_1, \dots, x_m \in \Lambda$ such that $\max_i \|y_i - x_i\| \le \eps$.  By the triangle inequality, for all $i \ne j$ we have 
\[\|x_i - x_j\| \ge \|y_i - y_j\| - 2\eps \ge \eta + \eps > \|x - x'\|.\]  Because $\psi$ is isotonic, we have $\|\psi(x_i) - \psi(x_j)\| \ge \xi$, so that $\psi(x_1), \dots, \psi(x_m)$ is a $\xi$-packing.
Therefore, there is a constant $C_2$ depending only on $d$ such that $m \le C_2 (\diam(Q)/\xi)^{d}$.  We conclude that $\xi \le (C_2/C_1)^{1/d} (\diam(Q)/\rho(V)) (\eta + \eps)$.
\end{proof}

For $V \subset \bbR^d$ and $h > 0$, let $V^{h} = \{x \in V : \exists y \in V \text{ s.t. } x \in B(y, h) \subset V\}$.  We note that $V^h$ is the complement of the $h$-convex hull of $V^\comp := \bbR^d \setminus V$ --- see \citep{MR2977397} and references therein.

\begin{lem} \label{lem:Lip-weak2}
In the context of \lemref{Lip}, if $\psi$ is only weakly isotonic, then there is $C \propto \diam(Q)$, such that for all $h > 0$,
\beq \label{Lip-weak2}
\|\psi(x) - \psi(x')\| \le C \big(\|x - x'\|/h + \sqrt{\eps/h}\big)^{1/d}, \quad \forall x \in \Lambda \cap V^{h}, \forall x' \in \Lambda.
\eeq
\end{lem}

\begin{proof}
Assume that $V^{h} \ne \emptyset$, for otherwise there is nothing to prove.  Take $x \in \Lambda \cap V^{h}$ and $x' \in \Lambda$ such that $\xi := \|\psi(x) - \psi(x')\| > 0$, and let $\eta = \|x - x'\|$.  
Because $\psi$ is bounded, it is enough to prove the result when $\eta, \eps < h/2$.  
Let $y \in V$ be such that $x \in B(y, h) \subset V$.  There is $y' \in B(y, h)$ such that $y \in [x y']$ and $\|x - y'\| \ge 2h/3$.  Define $u = (y'-x)/\|y'-x\|$.  
Let $z_0 = x$, and for $j \ge 1$, define $z_j = z_{j-1} + (\eta + 5 j \eps) u$.  
Let $k \ge 0$ be maximum such that $\sum_{j=1}^k (\eta + 5 j \eps) < h/2$.  
Since $k$ satisfies $k \eta + 5 k^2 \eps \ge h/2$, we have $k \ge \min(h/(4\eta), \sqrt{h/(10\eps)})$.
By construction, for all $j \in [k]$, $z_j \in [x y']$ and $B(z_j, 2\eps) \subset B(y,h)$.
Let $x_{-1} = x', x_0 = x$ and take $x_1, \dots, x_{k} \in \Lambda$ such that $\max_j \|x_j - z_j\| \le \eps$.  
By the triangle inequality, for $j=2,\dots,k$,
\[\|x_{j} - x_{j-1}\| \ge \|z_j - z_{j-1}\| - 2\eps \ge \|z_{j-1} - z_{j-2}\| + 3 \eps \ge \|x_{j-1} - x_{j-2}\| + \eps,\]
which implies by induction that 
\[\|x_{j} - x_{j-1}\| \ge \|x_1 - x_0\| + \eps \ge \|z_1 - z_0\| = \eta + 5\eps > \|x -x'\|.\]
By weak isotonicity, this implies that $\|\psi(x_{j}) - \psi(x_{j-1})\| \ge \|\psi(x) - \psi(x')\| = \xi$.
We also have, for any $i,j \in [k]$ such that $1 \le i \le j-2$, 
\[\|x_j - x_i\| \ge \|z_j - z_i\| - 2\eps \ge \|z_j - z_{j-1}\| + \eta + 5\eps - 2 \eps \ge \|x_j - x_{j-1}\| + \eta + \eps.\]
By weak isotonicity, this implies that $\|\psi(x_j) - \psi(x_i)\| \ge \|\psi(x_{j}) - \psi(x_{j-1})\|$ for all $0 \le i < j \le k$.  
Consequently, $(\psi(x_j) : j \in [k])$ forms a $\xi$-packing of $Q$.  Hence, $k \le C' (\diam(Q)/\xi)^{d}$, for some constant $C'$.  We conclude with the lower bound on $k$.  
\end{proof}

From this control on the modulus of continuity, we obtain a stronger version of Corollary~\ref{cor:C1}.  

\begin{thm} \label{thm:consistent}
Under the same conditions as Corollary~\ref{cor:C1}, we have the stronger conclusion that there is a sequence $(S_n)$ of similarities such that, for all $h > 0$, $\max_{x \in \Omega_n \cap U^{h}} \|\phi_n(x) - S_n(x)\| \to 0$ as $n \to \infty$.  
If in fact each $\phi_n$ is isotonic, then this remains true when $h = 0$.
\end{thm}

We remark that when $U$ is a connected union of a possibly uncountable number of open balls of radius at least $h > 0$, then $U = U^h$.  This covers the case of a finite union of open balls considered in \citep{klein}.  
We also note that, if $U$ is bounded and open, and $\partial U$ has bounded curvature, then there is $h>0$ such that $U = U^h$.  This follows from the fact that, in this case, $U^\comp$ has positive reach \citep{MR0110078}, and is therefore $h$-convex when $h$ is below the reach by\footnote{This proposition is stated for compact sets (which is not the case of $U^\comp$) but easily extends to the case where set is closed with compact boundary} \citep[Prop~1]{MR2977397}.  Moreover, our arguments can be modified to accommodate sets $U$ with boundaries that are only Lipschitz, by reasoning with wedges in \lemref{Lip-weak2}.

\thmref{consistent} now contains \citep[Th 4]{klein}, and extends it to weakly isotonic functions and to more general domains $U$.  
Overall, our proof technique is much simpler, shorter, and elementary.

Define $\eps_n = \delta_H(\Omega_n, U)$, which quantifies the density of $\Omega_n$ in $U$.
Because $\Omega_{n+1} \subset \Omega_n$ and $\Omega$ is dense in $U$, we have $\eps_n \searrow 0$ as $n \to \infty$.

\begin{proof}
Let $\phi$ be an accumulation point of $(\phi_n)$ for the pointwise convergence topology, meaning there is $N \subset \bbN$ infinite such that $\phi(x) = \lim_{n \in N} \phi_n(x)$ for all $x \in \Omega$.  We show that, in fact, the convergence is uniform.  

First, suppose that each $\phi_n$ is isotonic.  In that case, \lemref{Lip} implies the existence of a constant $C>0$ such that $\|\phi_n(x) - \phi_n(x')\| \le C (\|x - x'\| + \eps_n)$ for all $x,x' \in \Omega_n$, and for all $n$.  
Passing to the limit along $n \in N$, we get $\|\phi(x) - \phi(x')\| \le C \|x - x'\|$ for all $x,x' \in \Omega$.
(In fact, we already knew this from Corollary~\ref{cor:C1}, since we learned there that $\phi$ coincides with a similarity, and is therefore Lipschitz.)
Fix $\eps > 0$.
There is $m$ such that $\eps_m \le \eps$.
Then there is $m' \ge m$ such that $\max_{i \in [m]} \|\phi_n(x_i) - \phi(x_i)\| \le \eps$ for all $n \in N$ with $n \ge m'$.
For such an $n$, and $x \in \Omega_n$, let $i \in [m]$ be such that $\|x - x_i\| \le \eps_m$.  By the triangle inequality,
\[\begin{split}
\|\phi_n(x) - \phi(x)\| 
& \le \|\phi_n(x) - \phi_n(x_i)\| + \|\phi_n(x_i) - \phi(x_i)\| + \|\phi(x_i) - \phi(x)\| \\
& \le C (\|x - x_i\| + \eps_n) + \|\phi_n(x_i) - \phi(x_i)\| + C \|x_i - x\| \\
& \le C (\eps_m + \eps_n) + \eps + C \eps_m \le (3 C + 1) \eps. 
\end{split}\]
Since $x \in \Omega$ is arbitrary and $\eps$ can be taken as small as desired, this shows that the sequence $(\phi_n : n \in N)$ convergences uniformly to $\phi$ over $(\Omega_n : n \in N)$.

When the $\phi_n$ are only weakly isotonic, we use \lemref{Lip-weak2} to get a constant $C>0$ depending on $\diam(Q)$ and $h > 0$ such that $\|\phi_n(x) - \phi_n(x')\| \le C (\|x - x'\| + \sqrt{\eps_n})^{1/d}$ for all $x \in \Omega_n \cap U^{h}$ and all $x' \in \Omega_n$, and for all $n$.  
Passing to the limit along $n \in N$, we get $\|\phi(x) - \phi(x')\| \le C \|x - x'\|^{1/d}$ for all $x,x' \in \Omega$.
(In fact, $\|\phi(x) - \phi(x')\| \le C \|x - x'\|$ for all $x,x' \in \Omega$ from Corollary~\ref{cor:C1}, as explained above.)
The rest of the arguments are completely parallel.  We conclude that $(\phi_n : n \in N)$ convergences uniformly to $\phi$ over $(\Omega_n \cap U^h : n \in N)$.

Let $\cS$ denote the similarities of $\bbR^d$.
For any functions $\phi, \psi : \Omega \mapsto \bbR^d$, define $\delta_n(\phi, \psi) = \max_{x \in \Omega_n \cap U^h} \|\phi(x) - \psi(x)\|$, and also $\delta_n(\phi, \cS) = \inf_{S \in \cS} \delta_n(\phi, S)$.  
Our end goal is to show that $\delta_n(\phi_n, \cS) \to 0$ as $n \to \infty$.  Suppose this is not the case, so that there is $\eta > 0$ and $N \subset \bbN$ infinite such that $\delta_n(\phi_n, \cS) \ge \eta$ for all $n \in N$.  By Corollary~\ref{cor:C1}, there is $N_1 \subset N$ and $S \in \cS$ such that $S(x) = \lim_{n \in N_1} \phi_n(x)$ for all $x \in \Omega$.  As we showed above, the convergence is in fact uniform over $(\Omega_n \cap U^h : n \in N_1)$, meaning $\lim_{n \in N_1}\delta_n(\phi_n, S) = 0$.  At the same time, we have $\delta_n(\phi_n, S) \ge \delta_n(\phi_n, \cS) \ge \eta$.  We therefore have a contradiction.
\end{proof}


\subsection{Rates of convergence}
Beyond consistency, we are able to derive convergence rates.
We do so for the isotonic case, i.e., the quadruple comparison setting.
Recall that $\eps_n = \delta_H(\Omega_n, U)$.

\begin{thm} \label{thm:rates}
Consider the setting \eqref{setting} with $\phi_n$ isotonic. There is $C$ depending only on $(d, U)$, and a sequence of similarities $S_n$ such that $\max_{x \in \Omega_n} \|\phi_n(x) - S_n(x)\| \le C \diam(Q) \eps_n$.  If $U = U^h$ for some $h > 0$, then $C =  C'/\diam(U)$ where $C'$ is a function of $(d, h/\diam(U), \rho(U)/\diam(U))$.
\end{thm}


The proof of \thmref{rates} is substantially more technical than the previous results, and thus postponed to \secref{proofs}. 
Although \cite{klein} are not able to obtain rates of convergence, the proof of \thmref{rates} bares resemblance to their proof technique, and in particular, is also based on a result of \cite{alestalo2001isometric} on the approximation of {\em $\eps$-isometries}; see \lemref{eps-isometry}.  We will also make use of a related result of \cite{vestfrid2003linear} on the approximation of approximately midlinear functions; see \lemref{near-midpoint}.  
We mention that we know of a more elementary proof that only makes use of \citep{alestalo2001isometric}, but yields a slightly slower rate of convergence.  

We note that there is a constant $c$ depending only on $d$ such that $\eps_n \ge c n^{-1/d}$.  This is because $U$ being open, it contains an open ball, and this lower bound trivially holds for an open ball.  And such a lower bound is achieved when the $x_i$'s are roughly regularly spread out over $U$.  
If instead the $x_i$'s are iid uniform in $U$, and $U$ is sufficiently regular --- for example, $U = U^h$ for some $h > 0$ --- then $\eps_n = O(\log(n)/n)^{1/d}$, as is well-known.
This would give the rate, and we do not know whether it is optimal, even in dimension $d=1$.

\begin{rem}
We are only able to get a rate in $\sqrt{\eps_n}$ for the weakly isotonic case.  We can do so by adapting of the arguments underlying \thmref{rates}, but only after assuming that $U = U^h$ for some $h>0$ and resolving a few additional complications.
\end{rem}

\subsection{Ordinal embedding with local comparisons}

\cite{terada2014local} consider the problem of embedding an unweighted nearest-neighbor graph, which as we saw in the Introduction, is a special case of ordinal embedding.  Their arguments --- which, as we explained earlier, seem incomplete at the time of writing --- indicate that $K = K_n \gg n^{2/(2+d)} (\log n)^{d/(2+d)}$ is enough for consistently embedding a $K$-NN graph.  

We consider here a situation where we have more information, specifically, all the distance comparisons between $K$-nearest-neighbors.  Formally, this is the situation where 
\[
\cC_n = \Big\{(i,j,k,\ell): \delta_{ij} < \delta_{k\ell} \text{ and } \{j,k,\ell\} \subset N_{K_n}(i)\Big\},
\] 
where $N_K(i)$ denotes the set of the $K$ items nearest item $i$.  
If the items are points $\Omega_n = \{x_1, \dots, x_n\} \subset \bbR^d$, an exact ordinal embedding $\phi_n$ is only constrained to be locally weakly isotonic as we explain now.  We start by stating a standard result which relates a $K$-NN graph to an $r$-ball graph.

\begin{lem} \label{lem:graph}
Let $U \subset \bbR^d$ be bounded, connected and open, and such that $U = U^h$ for some $h > 0$.  Sample $x_1, \dots, x_n$ iid from a density $f$ supported on $U$ with (essential) range in $(0, \infty)$ strictly.
There is a constant $C$ such that, if $K := [n r^d] \ge C \log n$, then with probability tending to 1,
\[\neigh_{K/2}(x_i) \subset \{x_j : \|x_j - x_i\| \le r\} \subset \neigh_{2K}(x_i), \quad \forall i \in [n],\] 
where $\neigh_K(x_i)$ denotes the set of the $K$ points in $\{x_j : j \in [n]\}$ nearest $x_i$.  
\end{lem}

The proof is postponed to \secref{proofs} and only provided for completeness.
Therefore, assuming that $K \ge C \log n$, where $C$ is the constant of \lemref{graph}, we may equivalently consider the case where 
\[
\cC_n = \Big\{(i,j,k,\ell): \delta_{ij} < \delta_{k\ell} \text{ and } \max(\delta_{ij}, \delta_{ik}, \delta_{i\ell}) < r_n \Big\}, 
\]
for some given $r_n>0$.  An exact embedding $\phi_n : \Omega_n \mapsto \bbR^d$ in that case is isotonic on $\Omega_n \cap B(x, r_n)$ for any $x \in \Omega_n$.  We require in addition that 
\beq \label{outside}
\|x -x'\| < r_n \le \|x^\dag - x^\ddag\| \implies
\|\phi_n(x) - \phi_n(x')\| \le \|\phi_n(x^\dag) - \phi_n(x^\ddag)\|, \quad \forall x,x',x^\dag,x^\ddag \in \Omega_n.
\eeq 
This is a reasonable requirement since it is possible to infer it from $\cC_n$.  Indeed, for $k,\ell \in [n]$, we have $\delta_{k\ell} < r_n$ if, and only if, $(k,k,k,\ell) \in \cC_n$ or $(\ell, \ell, \ell, k) \in \cC_n$.  (Here we assume that $\delta_{ii} = 0$ for all $i$ and $\delta_{ij} > 0$ if $i \ne j$, as is the case for Euclidean distances.)  We can still infer this even if the quadruples in $\cC_n$ must include at least three distinct items.  Indeed, suppose $k,\ell \in [n]$ are such that there is no $i$ such that $(i,k,i,\ell) \in \cC_n$ or $(i,\ell,i,k) \in \cC_n$, then (a) $\delta_{ik} = \delta_{i\ell}$ for all $i$ such that $\max(\delta_{ik}, \delta_{i\ell}) < r_n$, or (b) $\delta_{k\ell} \ge r_n$.
Assume that $r_n \ge C \eps_n$ with $C > 0$ sufficiently large, so that situation (a) does not happen.  
Conversely, if $(k,\ell)$ is such that $\delta_{k\ell} < r_n$, then when (a) does not happen, there is $i$ such that $(i,k,i,\ell) \in \cC_n$ or $(i,\ell,i,k) \in \cC_n$.  

%

\begin{thm} \label{thm:local-quad}
Consider the setting \eqref{setting} and assume in addition that $U = U^h$ for some $h > 0$, and that $\phi_n$ is isotonic over balls of radius $r_n$ and satisfies \eqref{outside}.  
There is a constant $C>0$ depending only on $(d,h,\rho(U),\diam(U),\diam(Q))$ and similarities $S_n$ such that $\max_{x \in \Omega_n}\|\phi_n(x) - S_n(x)\| \le C \eps_n/r_n^2$.
\end{thm}

Assume the data points are generated as in \lemref{graph}.  In that case, we have $\eps_n = O(\log(n)/n)^{1/d}$ and \thmref{local-quad} implies consistency when $r_n \gg (\log(n)/n)^{1/2d}$.  By \lemref{graph}, this corresponds to the situation where we are provided with comparisons among $K_n$-nearest neighbors with $K_n \gg \sqrt{n \log n}$.  If the result of \cite{terada2014local} holds in all rigor, then this is a rather weak result.  


\subsection{Landmark ordinal embedding}

Inspired by \citep{jamieson2011low}, we consider the situation where there are landmark items indexed by $L_n \subset [n]$, and we are given all distance comparisons from any point to the landmarks.  Formally, with triple comparisons, this corresponds to the situation where 
\[
\cC_n = \Big\{(i,j,k) \in [n] \times L_n^2 : \delta_{ij} < \delta_{ik}\Big\}.
\]
If the items are points $\Omega_n = \{x_1, \dots, x_n\} \subset \bbR^d$, an exact ordinal embedding $\phi_n$ is only constrained to be weakly isotonic on the set of landmarks and, in addition, is required to respect the ordering of the distances from any point to the landmarks.   
The following is an easy consequence of \thmref{consistent}.

\begin{cor} \label{cor:landmark}
\thmref{consistent} remains valid in the landmark triple comparisons setting (meaning with $\phi_n$ as just described) as long as the landmarks become dense in $U$.
\end{cor}

\cite{jamieson2011low} study the number of triple comparisons that are needed for exact ordinal embedding.  With a counting argument, they show that at least $C n \log n$ comparisons are needed, where $C$ is a constant depending only on $d$. 
If we only insist that the embedding respects the comparisons that are provided, then Corollary~\ref{cor:landmark} implies that a  landmark design is able to be consistent as long as the landmarks become dense in $U$.  
This consistency implies that, as the sample size increases, an embedding that respects the landmark comparisons also respects all other comparisons approximately.
This is achieved with $O(n \ell_n^2 + \ell_n^3)$ triple comparisons, where $\ell_n := |L_n|$ is the number of landmarks, and the conditions of Corollary~\ref{cor:landmark} can be fulfilled with $\ell_n \to \infty$ at any speed, so that the number of comparisons is nearly linear in $n$.  

\begin{proof}
We focus on the weakly isotonic case, where we assume that $U = U^h$ for some $h > 0$.
Let $\Lambda_n = \{x_l : l \in L_n\}$ denote the set of landmarks.  Since $\Lambda_n$ becomes dense in $U$, meaning $\eta_n := \delta_H(\Lambda_n, U) \to 0$, by \thmref{consistent}, there is a sequence of similarities $S_n$ such that $\zeta_n := \max_{x \in \Lambda_n} \|\phi_n(x) - S_n(x)\| \to 0$.  
Now, for $x \in \Omega_n$, let $\tilde x \in \Lambda_n$ such that $\|x - \tilde x\| \le \eta_n$.  We have
\beq \label{landmark-proof1}
\|\phi_n(x) - S_n(x)\| \le \|\phi_n(x) - \phi_n(\tilde x)\| + \|\phi_n(\tilde x) - S_n(\tilde x)\| + \|S_n(\tilde x) - S_n(x)\|.
\eeq
The first term is bounded by $C \eta_n^{1/2d}$ by \lemref{Lip-weak2}, for some constant $C$.
The middle term is bounded by $\zeta_n$.  
For the third term, express $S_n$ in the form $S_n(x) = \beta_n R_n(x) + b_n$, where $\beta_n \in \bbR$, $R_n$ is an orthogonal transformation, and $b_n \in \bbR^d$.  Take two distinct landmarks $x^\dag, x^\ddag \in \Lambda_n$ such that $\|x^\dag - x^\ddag\| \ge \diam(U)/2$, which exist when $n$ is sufficiently large.  Since 
\[\|S_n(x^\dag) -S_n(x^\ddag)\| = \beta_n \|x^\dag - x^\ddag\| \ge \beta_n \diam(U)/2\] 
and, at the same time, 
\[\begin{split}
\|S_n(x^\dag) -S_n(x^\ddag)\| 
&\le \|S_n(x^\dag) -\phi_n(x^\dag)\| + \|\phi_n(x^\dag) -\phi_n(x^\ddag)\| + \|\phi_n(x^\ddag) -S_n(x^\ddag)\| \\
&\le \zeta_n + \diam(Q) + \zeta_n \le 2\diam(Q), \text{ eventually},
\end{split}\]
we have $\beta_n \le \bar \beta := 4 \diam(Q)/\diam(U)$.  Hence, the third term on the RHS of \eqref{landmark-proof1} is bounded by $\bar \beta\eta_n$.  Thus, the RHS of \eqref{landmark-proof1} is bounded by $C \eta_n^{1/2d} + \zeta_n + \bar \beta \eta_n$, which tends to 0 as $n \to \infty$.  This being valid for any $x \in \Omega_n$, we conclude.
\end{proof}

We remark that at the very end of the proof, we obtained a rate of convergence as a function of the density of the landmarks and the convergence rate implicit in \thmref{consistent}.  This leads to the following rate for the quadruple comparisons setting, which corresponds to the situation where 
\[
\cC_n = \Big\{(i,j,k) \in [n] \times L_n^2 : \delta_{ij} < \delta_{ik}\Big\} \bigcup \Big\{(i,j,k,\ell) \in L_n^4 : \delta_{ij} < \delta_{ik}\Big\}.  
\]
Here, $\phi_n$ is constrained to be isotonic on the set of landmarks and, as before, is required to respect the ordering of the distances from any data point to the landmarks.   

\begin{cor} \label{cor:landmark-quad}
Consider the setting \eqref{setting} in the landmark quadruple comparisons setting (meaning with $\phi_n$ as just described).  Let $\Lambda_n$ denote the set of landmarks and set $\eta_n = \delta_H(\Lambda_n, U)$.  There is a constant $C>0$ and a sequence of similarities $S_n$ such that $\max_{x \in \Omega_n} \|\phi_n(x) - S_n(x)\| \le C \eta_n$.
\end{cor}

\begin{proof}
The proof is parallel to that of Corollary~\ref{cor:landmark}.  Here, we apply \thmref{rates} to get $\zeta_n \le C_0 \eta_n$.  This bounds the second term on the RHS of \eqref{landmark-proof1}.  The first term is bounded by $C_1 \eta_n$ by \lemref{Lip}, while the third term is bounded by $\bar\beta \eta_n$ as before.   ($C_0, C_1$ are constants.)  
\end{proof}

\bigskip\noindent
{\bf Computational complexity.}
We now discuss the computational complexity of ordinal embedding with a landmark design.  The obvious approach has two stages.  In the first stage, the landmarks are embedded.  This is the goal of \citep{agarwal2007generalized}, for example.  Here, we use brute force. 

\begin{prp} \label{prp:embed-time}
Suppose that $m$ items are in fact points in Euclidean space and their dissimilarities are their pairwise Euclidean distances.  Then whether in the triple or quadruple comparisons setting, an exact ordinal embedding of these $m$ items can be obtained in finite expected time.
\end{prp}

\begin{proof}
The algorithm we discuss is very naive: we sample $m$ points iid from the uniform distribution on the unit ball, and repeat until the ordinal constraints are satisfied.  Since checking the latter can be done in finite time, it suffices to show that there is a strictly positive probability that one such sample satisfies the ordinal constraints.  Let $\cX_m$ denote the set of $m$-tuples $(x_1, \dots, x_m) \in B(0,1)$ that satisfy the ordinal constraints, meaning that $\|x_i - x_j\| < \|x_k - x_\ell\|$ when $(i,j,k,\ell) \in \cC$.  Seeing $\cX_n$ as a subset of $B(0,1)^m \subset \bbR^{dm}$, it is clearly open.  And sampling $x_1, \dots, x_m$ iid from the uniform distribution on $B(0,1)$ results in sampling $(x_1, \dots, x_m)$ from the uniform distribution on $B(0,1)^m$, which assigns a positive mass to any open set.
\end{proof}

In the second stage, each point that is not a landmark is embedded based on the order of its distances to the landmarks.  We quickly mention the work \cite{davenport2013lost}, who develops a convex method for performing this task.  Here, we are contented with knowing that this can be done, for each point, in finite time, function of the number of landmarks.
For example, a brute force approach starts by computing the Voronoi diagram of the landmarks, and iteratively repeats within each cell, creating a tree structure.  Each point that is not a landmark is placed by going from the root to a leaf, and choosing any point in that leaf cell, say its barycenter. 

Thus, if there are $\ell$ landmarks, the first stage is performed in expected time $F(\ell)$, and the second stage is performed in time $(n -\ell) G(\ell)$.  The overall procedure is thus computed in expected time $F(\ell) + (n-\ell)G(\ell)$.  

\begin{rem}
The procedure described above is not suggested as a practical means to perform ordinal embedding with a landmark design.  The first stage, described in \prpref{embed-time}, has finite expected time, but likely {\em not} polynomial in the number of landmarks.  For a practical method, we can suggest the following:
\benum
\item Embed the landmarks using the method of \cite{agarwal2007generalized} (which solves a semidefinite program) or the method of \cite{terada2014local} (which uses an iterative minimization-majorization strategy).  
\item Embed the remaining points using the method of \cite{davenport2013lost} (which solves a quadratic program).
\eenum
\end{rem}
Although practical and reasonable, we cannot provide any theoretical guarantees for this method.


\section{More proofs} \label{sec:proofs}

In this section we gather the remaining proofs and some auxiliary results.
We introduce some additional notation and basic concepts.  
For $z_1, \dots, z_m \in \bbR^d$, let $\aff(z_1, \dots z_m)$ denote their affine hull, meaning the affine subspace they generate in $\bbR^d$.
For a vector $x$ in a Euclidean space, let $\|x\|$ denote its Euclidean norm.  For a matrix $M \in \bbR^{p \times q}$, let $\|M\|$ denote its usual operator norm, meaning, $\|M\| = \max\{\|Mx\| : \|x\| \le 1\}$ and $\|M\|_F = \sqrt{\tr(M^\top M)}$ its Frobenius norm.

\medskip\noindent {\bf Regular simplexes.}
These will play a central role in our proofs.
We say that $z_1, \dots, z_m \in \bbR^d$, with $m \ge 2$, form a regular simplex if their pairwise distances are all equal.  
We note that, necessarily, $m \le d+1$, and that regular simplexes in the same Euclidean space and with same number of (distinct) nodes $m$ are similarity transformations of each other --- for example, segments ($m=2$), equilateral triangles ($m=3$), tetrahedron ($m=4$).
By recursion on the number of vertices, $m$, it is easy to prove the following.  

\begin{lem} \label{lem:simplex}
Let $z_1, \dots, z_m$ form a regular simplex with edge length~1 and let $\mu$ 
denote the barycenter of $z_1, \dots, z_{m}$.  Then $\|\mu - z_i\| = \sqrt{(m-1)/2m}$, and if $z, z_1, \dots, z_{m}$ form a regular simplex, then $\|z - \mu\| = \sqrt{(m+1)/2m}$.  (In dimension $m$, there are exactly two such points $z$.) 
\end{lem}

\subsection{Proof of \thmref{infinite}} \label{sec:proof-thm-infinite}

We assume $d \ge 2$.  See \citep{klein} for the case $d = 1$.  We divide the proof into several parts.

\medskip\noindent {\bf Continuous extension.}
\lemref{Lip-weak2} implies that $\phi$ is locally uniformly continuous.  Indeed, take $x_0 \in \Omega$ and let $r > 0$ such that $B(x_0, r) \subset U$ and $\phi$ is weakly isotonic on $B(x_0, r) \cap \Omega$.  Applying \lemref{Lip-weak2} with $V = B(x_0,r)$ and $\Lambda = \Omega \cap B(x_0, r)$ --- so that $\delta_H(\Lambda, V) = 0$ because $\Lambda$ is dense in $V$ --- and noting that $V^r = V$, yields a constant $C_r>0$ such that $\|\phi(x) - \phi(x')\| \le C_r \|x - x'\|^{1/d}$, for all $x, x' \in \Omega \cap B(x_0, r)$.  Being locally uniformly continuous, we can uniquely extend $\phi$ to a continuous function on $U$, also denoted by $\phi$.  By continuity, this extension is locally weakly isotonic on $U$.  

\medskip\noindent {\bf Isosceles preservation.}
\cite{sikorska2004mappings} say that a function $f:V \subset \bbR^d \to \bbR^d$ preserves isosceles triangles if 
\[
\|x - y\| = \|x - z\| \implies \|f(x) - f(y)\| = \|f(x) - f(z)\|, \quad \forall x,y,z \in V.
\]
In our case, by continuity, we also have that $\phi$ preserves isosceles triangles locally.  Indeed, for the sake of pedagogy, let $u \in U$ and $r > 0$ such that $B(u,r) \subset U$ and $\phi$ is weakly isotonic on $B(u,r)$.  Take $x,y,z \in B(u,r/2)$ be such that $\|x - y\| = \|x - z\|$.  For $t \in \bbR$, define $z_t = (1-t)x + t z$.  Let $t > 1$ such that $z_t \in B(u,r)$.  Because $\|x - y\| < t \|x - z\| = \|x - z_t\|$, we have $\|\phi(x) - \phi(y)\| \le \|\phi(x) - \phi(z_t)\|$.  Letting $t \searrow 1$, we get $\|\phi(x) - \phi(y)\| \le \|\phi(x) - \phi(z)\|$ by continuity of $\phi$.  Since $y$ and $z$ play the same role, the converse inequality is also true, and combined, yield an equality.


\medskip\noindent {\bf Midpoint preservation.}
Let $V \subset \bbR^d$ be convex.  We say that a function $f:V \mapsto \bbR^d$ preserves midpoints if 
\[
f\Big(\frac{x+y}2\Big) = \frac{f(x) + f(y)}2, \quad \forall x,y \in V.
\]
We now show that $\phi$ preserves midpoints, locally.  
\cite{klein} also do that, however, our arguments are closer to those of \cite{sikorska2004mappings}, who make use of regular simplexes. 
The important fact is that a function that preserves isosceles preserves regular simplexes. 
Let $u \in U$ and $r > 0$ such that $B(u,r) \subset U$ and $\phi$ preserves isosceles on $B(u,r)$.  Take $x,y \in B(u,r/2)$, and let $\mu = (x+y)/2$.  Let $z_1, \dots, z_d$ form a regular simplex with barycenter $\mu$ and side length $s$, and such that $\|x - z_i\| = s$ for all $i$.  In other words, $x, z_1, \dots, z_d$ forms a regular simplex placed so that $\mu$ is the barycenter of $z_1, \dots, z_d$.
By symmetry, $y, z_1, \dots, z_d$ forms a regular simplex also.  
By \lemref{simplex}, we have $\|z_i - \mu\|/\|x - \mu\| = \sqrt{(d-1)/(d+1)}$, so that $z_1, \dots, z_d \in B(\mu, r/2) \subset B(u, r)$, by the triangle inequality and the fact that $\|x - \mu\| < r/2$.
Hence, $\phi(x), \phi(z_1), \dots, \phi(z_d)$ and $\phi(y), \phi(z_1), \dots, \phi(z_d)$ are regular simplexes.  If one of them is singular, so is the other one, in which case $\phi(x) = \phi(y) = \phi(\mu)$.  Otherwise, necessarily $\phi(x)$ is the symmetric of $\phi(y)$ with respect to $\aff(\phi(z_1), \dots, \phi(z_d))$; the only other possibility would be that $\phi(x) = \phi(y)$, but in that case we would still have that $\phi(z_i) = \phi(x)$ for all $i \in [d]$, since  $\|x - z_i\|/\|x - \mu\| = \sqrt{2d/(d+1)}$ by \lemref{simplex} --- implying that $\|x - z_i\| < \|x - y\|$ --- and $\phi$ is weakly isotonic in that neighborhood.  So assume that $\phi(x)$ is the symmetric of $\phi(y)$ with respect to $\aff(\phi(z_1), \dots, \phi(z_d))$.  For $a \in \{x,y,\mu\}$, $\|a - z_i\|$ is constant in $i$, and therefore so is $\|\phi(a) - \phi(z_i)\|$, so that $\phi(a)$ belongs to the line of points equidistant to $\phi(z_1), \dots, \phi(z_d)$.  This implies that $x,y,\mu$ are collinear.  And because $\|\mu - x\| = \|\mu - y\|$, we also have $\|\phi(\mu) - \phi(x)\| = \|\phi(\mu) - \phi(y)\|$, so that $\phi(\mu)$ is necessarily the midpoint of $\phi(x)$ and $\phi(y)$. 

\medskip\noindent {\bf Conclusion.}
We arrived at the conclusion that $\phi$ can be extended to a continuous function on $U$ that preserves midpoints locally.  
We then use the following simple results in sequence:  
with \lemref{midpoint}, we conclude that $\phi$ is locally affine;
with \lemref{affine}, we conclude that $\phi$ is in fact affine on $U$;
and with \lemref{iso}, we conclude that $\phi$ is in fact a similarity on $U$.

\begin{lem} \label{lem:midpoint}
Let $V$ be a convex set of a Euclidean space and let $f$ be a continuous function on $V$ with values in a Euclidean space that preserves midpoints.  Then $f$ is an affine transformation.  
\end{lem}

\begin{proof}
This result is in fact well-known, and we only provide a proof for completeness.  It suffices to prove that $f$ is such that $f((1-t) x + t y) = (1-t) f(x) + t f(y)$ for all $x,y \in V$ and all $t \in [0,1]$.  Starting with the fact that this is true when $t = 1/2$, by recursion we have that this is true when $t$ is dyadic, meaning, of the form $t = k 2^{-j}$, where $j \ge 1$ and $k \le 2^j$ are both integers.  Since dyadic numbers are dense in $[0,1]$, by continuity of $f$, we deduce the desired property.
\end{proof}

\begin{lem} \label{lem:affine}
A locally affine function over an open and connected subset of a Euclidean space is the restriction of an affine function over the whole space.
\end{lem}

\begin{proof}
Let $U$ be the domain and $f$ the function.  Cover $U$ with a countable number of open balls $B_i, i \in I$ such that $f$ coincides with an affine function $f_i$ on $B_i$.  Take $i, j \in I$ distinct.  Since $U$ is connected, there must be a sequence $i = k_1, \dots, k_m = j$, all in $I$, such that $B_{k_s} \cap B_{k_{s+1}} \ne \emptyset$ for $s \in  [m-1]$.  Since $B_{k_s} \cap B_{k_{s+1}}$  is an open set, we must have $f_{k_s} = f_{k_{s+1}}$, and this being true for all $s$, it implies that $f_i = f_j$.  
\end{proof}

\begin{lem} \label{lem:iso}
An affine function that preserves isosceles locally is a similarity transformation.
\end{lem}

\begin{proof}
Let $f$ be an affine function that preserves isosceles in an open ball.  Without loss of generality, we may assume that the ball is $B(0, 2)$ and that $f(0) = 0$ (so that $f$ is linear). 
Fix $u_0 \in \partial B(0,1)$ and let $a = \|f(u_0)\|$.  
Take $x \in \bbR^d$ different from 0 and let $u = x/\|x\|$.
We have $\|f(x)\|/\|x\| = \|f(u)\| = \|f(u) - f(0)\| = \|f(u_0) - f(0)\| = \|f(u_0)\| = a$.  
Hence, $\|f(x)\| = a \|x\|$, valid for all $x \in \bbR^d$, and $f$ being linear, this implies that $f$ is a similarity.
\end{proof}

\subsection{Auxiliary results}

We list here a number of auxiliary results that will be used in the proof of \thmref{rates}.

The following result is a perturbation bound for trilateration, which is the process of locating a point based on its distance to landmark points.
For a real matrix $Z$, let $\sigma_k(Z)$ denote its $k$-th largest singular value.

\begin{lem} \label{lem:trilateration}
Let $z_1, \dots, z_{d+1} \in \bbR^d$ such that $\aff(z_1, \dots, z_{d+1}) = \bbR^d$ and let $Z$ denote the matrix with columns $z_1, \dots, z_{d+1}$.  Consider $p, q \in \bbR^d$ and define $a_i = \|p - z_i\|$ and $b_i = \|q - z_i\|$ for $i \in [d+1]$.  Then 
\[
\|p - q\| \le \frac12 \sqrt{d} \, \sigma_d(Z)^{-1} \max_i |a_{d+1}^2 - a_i^2 - b_i^2 + b_{d+1}^2| \le \sqrt{d} \, \sigma_d(Z)^{-1} \max_i |a_i^2 - b_i^2|.
\]
\end{lem}

\begin{proof}
Assume without loss of generality that $z_{d+1} = 0$.
In that case, note that $a_{d+1} = \|p\|$ and $b_{d+1} = \|q\|$.
Also, redefine $Z$ as the matrix with columns $z_1, \dots, z_{d}$, and note that the first $d$ singular values remain unchanged.
Since $\aff(z_1, \dots, z_{d+1}) = \bbR^d$, there is $\alpha = (\alpha_1, \dots, \alpha_d) \in \bbR^{d}$ and $\beta = (\beta_1, \dots, \beta_d) \in \bbR^d$ such that $p = \sum_{i \in [d]} \alpha_i z_i = Z \alpha$ and $q = \sum_{i \in [d]} \beta_i z_i = Z \beta$.  For $p$, we have $a_i^2 = \|p - z_i\|^2 = \|p\|^2 + \|z_i\|^2 - 2 z_i^\top Z \alpha$ for all $i \in [d]$, or in matrix form, $Z^\top Z \alpha = \frac12 u$, where $u = (u_1, \dots, u_d)$ and $u_i = a_{d+1}^2 - a_i^2 + \|z_i\|^2$.  Similarly, we find $Z^\top Z \beta = \frac12 v$, where $v = (v_1, \dots, v_d)$ and $v_i = b_{d+1}^2 - b_i^2 + \|z_i\|^2$.  Hence, we have
\[\begin{split}
\|Z^\top(p - q)\| 
&= \|Z^\top Z \alpha - Z^\top Z \beta\| = \frac12 \|u - v\| = \frac12 \sqrt{\sum_{i \in [d]} (a_{d+1}^2 - b_{d+1}^2 - a_i^2 + b_i^2)^2} \\
&\le \frac12 \sqrt{d} \, \max_i |a_{d+1}^2 - a_i^2 - b_i^2 + b_{d+1}^2| \le \sqrt{d} \, \max_i |a_i^2 - b_i^2|.
\end{split}\]
Simultaneously, $\|Z^\top(p - q)\| \ge \sigma_d(Z) \|p - q\|$.
Combining both inequalities, we conclude.
\end{proof}

For $\eta \in [0,1)$, we say that $z_1, \dots, z_m \in \bbR^d$ form an $\eta$-approximate regular simplex if 
\[\min_{i \ne j} \|z_i - z_j\| \ge (1-\eta) \max_{i \ne j} \|z_i - z_j\|.\]

\begin{lem} \label{lem:approx-simplex}
Let $z_1, \dots, z_m$ form an $\eta$-approximate regular simplex with maximum edge length~$\lambda$ achieved by $\|z_1 -z_2\|$.  There is a constant $C_m$ and $z'_1, \dots, z'_m \in {\rm Aff}(z_1, \dots, z_m)$ with $z'_1 = z_1$ and $z'_2 = z_2$ and forming a regular simplex with edge length~$\lambda$, such that $\max_i \|z'_i - z_i\| \le \lambda C_m\eta$.
\end{lem}

\begin{proof}
By scale equivariance, we may assume that $\lambda = 1$.  
We use an induction on $m$.  
In what follows, $C_m, C_m', C_m''$, etc, are constants that depend only on $m$.
For $m=2$, the statement is trivially true.  
Suppose that it is true for $m \ge 2$ and consider an $\eta$-approximate regular simplex $z_1, \dots, z_{m+1} \in \bbR^d$ with maximum edge length~1.  
By changing $d$ to $m$ if needed, without loss of generality, assume that ${\rm Aff}(z_1, \dots, z_{m+1}) = \bbR^d$.
In that case, $z_1, \dots, z_m$ is an $\eta$-approximate regular simplex with maximum edge length achieved by $\|z_1 -z_2\| = 1$, and by the inductive hypothesis, this implies the existence of $z'_1, \dots, z'_m \in A := {\rm Aff}(z_1, \dots, z_m)$ with $z'_1 = z_1$ and $z'_2 = z_2$ and forming a regular simplex of edge length~1, such that $\max_{i \in [m]} \|z'_i - z_i\| \le C_m \eta$ for some constant $C_m$.  
%
Let $p$ be the orthogonal projection of $z_{m+1}$ onto $A$.
Before continuing, let $\cP$ be the set of such $p$ obtained when fixing $z'_1, \dots, z'_m$ and then varying $z_i \in B(z'_i, C_m \eta)$ for $i \in [m]$ and $z_{m+1}$ among the points that make an $\eta$-approximate regular simplex with $z_1, \dots, z_m$.  
Let $\mu'$ be the barycenter of $z'_1, \dots, z'_m$ and note that $\mu' \in \cP$. 
Now, set $\delta = \|z_{m+1} - p\|$.  By the Pythagoras theorem, we have $\|p - z_i\|^2 = \|z_{m+1} - z_i\|^2 - \delta^2$, with $1-\eta \le \|z_{m+1} - z_i\| \le 1$, so that $0 \le 1-\delta^2 - \|p - z_i\|^2 \le 2\eta$.  By the triangle inequality, $\big|\|p - z'_i\| - \|p - z_i\| \big| \le \|z_i - z_i'\| \le C_m\eta$, so that
\[\begin{split}
\big|\|p - z'_i\|^2 - \|p - z_i\|^2 \big| 
&= \big|\|p - z'_i\| - \|p - z_i\| \big| \big(\|p - z'_i\| - \|p - z_i\|\big) \le \|z_i - z_i'\|\|z_i - z_i'\|
\\&\le C_m \eta \big(2 + C_m \eta\big) \le C_m' \eta,
\end{split}\]
using the fact that $\|p - z_i\| \le \|z_{m+1} - z_i\| \le 1$. 
Hence, 
\[\cP \subset \Big\{q:  \|q - z'_i\|^2 = 1 -\delta^2 \pm C_m'' \eta, \forall i \in [m]\Big\}.\]  
Since $\mu' \in \cP$, we must therefore have $\|p - z'_i\|^2 = \|\mu' - z'_i\|^2 \pm 2 C_m'' \eta$.
By \lemref{trilateration}, this implies that $\|p - \mu'\| \le \sqrt{m-1} \, \sigma_{m-1}^{-1}([z'_1 \cdots z'_m]) 2 C_m'' \eta =: C_m''' \eta$.
Let $z'_{m+1}$ be on the same side of $A$ as $z_{m+1}$ and such that $z'_1,\dots,z'_m,z'_{m+1}$ form a regular simplex.  
Note that $\mu'$ is the orthogonal projection of $z'_{m+1}$ onto $A$.
By the Pythagoras theorem, applied multiple times, we obtain the following.  First, we have 
\[\begin{split}
\|z'_{m+1} - z_{m+1}\|^2
&= \|z'_{m+1} - \mu' + \mu' -p + p - z_{m+1}\|^2 \\
&= \|z'_{m+1} - \mu'\|^2 - 2 (z'_{m+1} - \mu')^\top (z_{m+1} - p) + \|p - z_{m+1}\|^2 + \|\mu' -p\|^2 \\
&= (\|z'_{m+1} - \mu'\| - \|p - z_{m+1}\|)^2 + \|\mu' -p\|^2,
\end{split}\]
because $z'_{m+1} - \mu'$ and $z_{m+1}-p$ are orthogonal to $A$, and therefore parallel to each other and both orthogonal to $\mu' - p$. 
For the second term, we already know that $\|\mu' -p\| \le C_m''' \eta$, while the first term is bounded by $(2 C_m'' + 2)^2 \eta^2$ since, on the one hand, 
\[
\|z'_{m+1} - \mu'\|^2 = \|z'_{m+1} - z'_1\|^2 - \|\mu' - z'_1\|^2 = 1 - \|\mu' - z'_1\|^2
\]
while, on the other hand, 
\[
\|p - z_{m+1}\|^2 = \|z_{m+1} - z_1\|^2 - \|p - z_1\|^2 = 1 \pm 2 \eta - \|p - z_1\|^2,
\]
and we know that $\|\mu' - z'_1\|^2 = \|p - z_1\|^2 \pm 2 C_m'' \eta$.
Hence, we find that $\|z'_{m+1} - z_{m+1}\|^2 \le C^2_{m+1} \eta^2$ for some constant $C_{m+1}$ function of $m$ only.
This shows that the induction hypothesis holds for $m+1$.
\end{proof}

\begin{lem} \label{lem:simplex-value}
There are constants $C_m, C'_m > 0$ such that, if $z_1, \dots, z_m$ form an $\eta$-approximate regular simplex with maximum edge length~$\lambda$, then $\sigma_{m-1}([z_1 \cdots z_m]) \ge \lambda C_m (1 -C'_m \eta)$.
\end{lem}

\begin{proof}
By scale equivariance, we may assume that $\lambda = 1$.  
By \lemref{approx-simplex}, there is a constant $C''_m$ and $z'_1, \dots, z'_m \in \aff(z_1, \dots, z_m)$ forming a regular simplex with edge length 1 such that $\max_i \|z'_i - z_i\| \le C''_m \eta$.  
By Weyl's inequality \citep[Cor~7.3.8]{MR1084815}, $\sigma_{m-1}(Z) \ge \sigma_{m-1}(Z') - \|Z - Z'\|$.  On the one hand, $\sigma_{m-1}(Z')$ is a positive constant depending only on $m$, while on the other hand, $\|Z - Z'\| \le \|Z -Z'\|_F = \sqrt{\sum_i \|z_i - z_i'\|^2} \le \sqrt{m} C''_m \eta$.
\end{proof}

\begin{lem} \label{lem:barycenter}
Let $z_1, \dots, z_m$ form an $\eta$-approximate regular simplex with maximum edge length~$\lambda$ and barycenter $\mu$.  
Let $p \in \aff(z_1, \dots, z_m)$ and define $\gamma = \max_i \|p - z_i\|^2 - \min_i \|p - z_i\|^2$.  There is a constant $C_m \ge 1$ depending only on $m$ such that $\|p - \mu\| \le C_m \lambda \gamma$ when $\eta \le 1/C_m$.
\end{lem}

\begin{proof}
By scale equivariance, we may assume that $\lambda = 1$.  
By \lemref{trilateration}, we have 
\[
\|p - \mu\| \le \frac12 \sqrt{m-1} \, \sigma_{m-1}^{-1}([z_1 \cdots z_m]) \max_i \big|\|p - z_m\|^2 - \|p - z_i\|^2\big|.
\]
By \lemref{simplex-value}, there is a constant $C_m'$ such that $\sigma_{m-1}^{-1}([z_1 \cdots z_m]) \le C_m'$ when $\eta \le 1/C_m'$.
And we also have $\max_i \big|\|p - z_m\|^2 - \|p - z_i\|^2\big| \le \gamma$.
From this, we conclude.
\end{proof}

%

\begin{lem} \label{lem:near-sim}
Let $\psi : \Lambda \mapsto Q$ be isotonic, where $\Lambda, Q \subset \bbR^d$.
Let $v \in \bbR^d$ and $r > 0$, and set  $\eps = \delta_H(\Lambda, B(v, r))$.    
There is $C \propto \diam(Q)/r$ such that, for all $x,x',x^\dag,x^\ddag \in \Lambda$ with $x, x' \subset B(v, 3r/4)$ and for all $\eta \in (0,r/4-2\eps)$,
\beq \label{near-sim}
\|x - x'\| = \|x^\dag - x^\ddag\| \pm \eta \implies \|\psi(x) - \psi(x')\| = \|\psi(x^\dag) - \psi(x^\ddag)\| \pm C(\eta + \eps).
\eeq
\end{lem}

\begin{proof}
Let $\xi = \|x - x'\|$ and $\xi^\dag = \|x^\dag - x^\ddag\|$.  
Suppose that $\xi < \eta + 2\eps$, which implies that $\xi^\dag < 2\eta+2\eps$.  In that case, \lemref{Lip} --- where the constant there is denoted here by $C_1 \propto \diam(Q)/r$ --- yields $\|\psi(x) - \psi(x')\| \le C_1 (\xi + \eps) \le C_1 (\eta + 3\eps)$ and, similarly, $\|\psi(x^\dag) - \psi(x^\ddag)\| \le C_1 (2\eta + 3\eps)$.  This proves \eqref{near-sim}.  Henceforth, we assume that $\xi \ge \eta + 2\eps$.

First assume that $\xi > \xi^\dag$.
In that case, we immediately have $\|\psi(x) - \psi(x')\| \ge \|\psi(x^\dag) - \psi(x^\ddag)\|$.
For the reverse, let $y_t = (1-t) x + t x'$, and note that $\|y_t - x\| = t \xi$.  Take $t = 1 - (\eta + 2\eps)/\xi$ and note that $t \in [0,1]$, so that $y_{t} \in [x x'] \subset B(v,r)$, and therefore there is $x^\star \in \Lambda$ such that $\|x^\star - y_{t}\| \le \eps$.
We have $\|x^\star - x\| \le \|y_{t} - x\| + \|x^\star - y_{t}\| \le \xi - \eta - \eps < \xi^\dag$, so that $\|\psi(x) - \psi(x^\star)\| \le \|\psi(x^\dag) - \psi(x^\ddag)\|$.  Applying the triangle inequality and \lemref{Lip}, we then have 
\[
\begin{split}
\|\psi(x) - \psi(x^\star)\| 
&\ge \|\psi(x) - \psi(x')\| - \|\psi(x') - \psi(x^\star)\| \\
&\ge \|\psi(x) - \psi(x')\| - C_1 (\|x' - x^\star\| + \eps),
\end{split}
\]
with $\|x' - x^\star\| \le \|x' - y_{t}\| + \|y_{t}-x^\star\| \le \eta + 3\eps$.

When $\xi < \xi^\dag$, we choose $t = 1 + (\eta + 2\eps)/\xi$.  Because $x, x' \subset B(v, 3r/4)$, we still have $y_t \in B(v,r)$ because of the constraint on $\eta$.  The remaining arguments are analogous.

When $\xi =  \xi^\dag$, repeating what we just did both ways and with $\eta = 0$ yields the result.
\end{proof}

\begin{lem} \label{lem:diam-bound}
Consider $\psi : \Lambda \mapsto \bbR^d$ isotonic, where $\Lambda \subset \bbR^d$.  Let V denote the convex hull of $\Lambda$.  Set $\eps = \delta_H(\Lambda, V)$ and $c = \diam(\psi(\Lambda))/5\diam(\Lambda)$.  Then $\|\psi(x) - \psi(x')\| \ge c \|x -x'\|$ for all $x, x' \in \Lambda$ such that $\|x -x'\| \ge 4 \eps$. 
\end{lem}

\begin{proof}
We first prove that, if $c>0$ and $\eta \ge 4 \eps$ are such that $\|\psi(x) - \psi(x')\| \le c \eta$ for all $x,x' \in \Lambda$ with $\|x - x'\| < \eta$, then $\diam(\psi(\Lambda)) < c (4 \diam(\Lambda) + \eta)$.  
Indeed, take $x, x' \in \Lambda$.  
Let $u = (x'-x)/\|x'-x\|$ and $L = \|x-x'\|$, and define $y_j = x + s_j u$ where $s_j = j (\eta - 3\eps)$ for $j = 0, \dots, J := \lfloor L/(\eta - 3\eps) \rfloor$, and then let $s_{J+1} = L$.  
By construction, $y_j \in [xx'] \subset V$, with $y_0 = x$ and $y_{J+1} = x'$.
Let $x_j \in \Lambda$ be such that $\|x_j - y_j\| \le \eps$, with $x_0 = x$ and $x_{J+1} = x'$.  
By the triangle inequality, $\|x_{j+1} - x_j\| \le \|y_{j+1} - y_j\| + 2\eps = s_{j+1} - s_j + 2 \eps < \eta$.  Hence,
\[
\|\psi(x) - \psi(x')\| \le \sum_{j=0}^{J} \|\psi(x_j) - \psi(x_{j+1})\| \le (J+1) c \eta \le c \frac{L \eta}{\eta - 3\eps} + c \eta < c (4 \diam(\Lambda) + \eta),
\]
since $\eta - 3\eps \ge \eta - 3\eta/4 = \eta/4$ and $L \le \diam(\Lambda)$.

Now assume that $\psi$ is isotonic and suppose that $\|\psi(x) - \psi(x')\| < c \|x -x'\|$ for some $x,x' \in \Lambda$ such that $\eta := \|x -x'\| \ge 4 \eps$.  Then we have $\|\psi(x^\dag) - \psi(x^\ddag)\| \le c \eta$ when $x^\dag, x^\ddag \in \Lambda$ satisfy $\|x^\dag -x^\ddag\| < \eta$.  We just showed that this implies that $\diam(\psi(\Lambda)) < c (4 \diam(V) + \eta)$, and we conclude using the fact that $\eta \le \diam(\Lambda)$.
\end{proof}

The following result is on 1-nearest neighbor interpolation.  
\begin{lem} \label{lem:1nn}
Let $\Lambda$ be a subset of isolated points in $V \subset \bbR^d$ and set $\eps = \delta_H(\Lambda, V)$.  For any function $\psi : \Lambda \mapsto \bbR^d$, define its 1-nearest neighbor interpolation as $\hat\psi : V \mapsto \bbR^d$ as 
\beq \label{1NN}
\hat\psi(y) = \frac1{|N_\Lambda(y)|} \sum_{x \in N_\Lambda(y)} \psi(x), \quad N_\Lambda(y) := \argmin_{x \in \Lambda} \|x - y\|.
\eeq
Consider the modulus of continuity of $\psi$, which for $\eta > 0$ is defined as $\omega(\eta) = \sup\{ \|\psi(x) - \psi(x')\| : x, x' \in \Lambda, \|x - x'\| \le \eta\}$.  Then the modulus of continuity of $\hat\psi$, denoted $\hat\omega$, satisfies $\hat\omega(\eta) \le \omega(\eta + 2 \eps)$.
Moreover, for any $y,y' \in V$ and any $x,x' \in \Lambda$ such that $\|x-y\| \le \eps$ and $\|x'-y'\| \le \eps$,
\[
\|\hat\psi(y) - \hat\psi(y')\| = \|\psi(x) - \psi(x')\| \pm 2\omega(2\eps).
\] 
\end{lem}

\begin{proof}
Fix $\eta > 0$ and take $y, y' \in V$ such that $\|y - y'\| \le \eta$.  We have $\|x - y\| \le \eps$ for all $x \in N_\Lambda(y)$  and $\|x' - y'\| \le \eps$ for all $x' \in N_\Lambda(y')$, so that $\|x - x'\| \le \|y - y'\| + 2 \eps$ for all such $x$ and $x'$, by the triangle inequality.  Therefore, 
\[\begin{split}
\|\hat\psi(y) - \hat\psi(y')\| 
&\le \sup\Big\{\|\psi(x) - \psi(x')\| : x \in N_\Lambda(y), x' \in N_\Lambda(y')\Big\} \\
&\le \sup\Big\{\|\psi(x) - \psi(x')\| : x, x' \in \Lambda, \|x - x'\| \le \eta + 2\eps\Big\} = \omega(\eta + 2\eps).
\end{split}\]
Since this is true for all $y, y' \in V$ such that $\|y - y'\| \le \eta$, we conclude that $\hat\omega(\eta) \le \omega(\eta + 2\eps)$.

For the second part of the lemma, we have
\[
\|\hat\psi(y) - \hat\psi(y')\|  = \|\psi(x) - \psi(x')\| \pm \|\hat\psi(y) - \psi(x)\| \pm \|\hat\psi(y') - \psi(x')\|,
\]
where the second term is bounded by
\[\begin{split}
\|\hat\psi(y) - \psi(x)\| 
&\le \sup\Big\{\|\psi(\tilde x) - \psi(x)\| : \tilde x \in N_\Lambda(y)\Big\} \\
&\le \sup\Big\{\|\psi(\tilde x) - \psi(x)\| : \|\tilde x - x\| \le 2 \eps\Big\} \le \omega(2\eps),
\end{split}\]
using the fact that $\|\tilde x - x\| \le \|\tilde x - y\| + \|y - x\| \le 2\eps$, and similarly for the third term.
\end{proof}

Let $V \subset \bbR^d$ be convex.  In our context, we say that $f : V \mapsto \bbR^d$ is {\em $\eta$-approximately midlinear} if 
\[
\Big\|f\Big(\frac{x+y}2\Big) - \frac12 (f(x) + f(y))\Big\| \le \eta, \quad \forall x,y \in V.
\]

\begin{lem} \label{lem:near-midpoint} 
Let $V \subset \bbR^d$ be star-shaped with respect to some point in its interior.  There is a constant $C$ depending only on $V$ such that, for any $\eta$-approximately midlinear function $f: V \mapsto \bbR^d$, there is a affine function $T : \bbR^d \mapsto \bbR^d$ such that $\sup_{x \in V} \|f(x) - T(x)\| \le C \eta$.
\end{lem}

Note that, if $V$ is a ball, then by invariance considerations, $C$ only depends on $d$.

\begin{proof}
This is a direct consequence of \citep[Th~1.4]{vestfrid2003linear}.
\end{proof}

We say that $f : V \subset \bbR^d \mapsto \bbR^d$ is an {\em $\eps$-isometry} if 
\[
\|x - y\| -\eps \le \|f(x) - f(y)\| \le \|x - y\| +\eps, \quad \forall x, y \in V. 
\]
For a set $V \subset \bbR^d$, define its thickness as 
\[
\theta(V) = \inf\big\{ \diam(u^\top V) : u \in \bbR^d, \|u\| = 1\big\}.
\]
Recalling the definition of $\rho$ in \eqref{rho}, we note that $\theta(V) \ge \rho(V)$, but that the two are distinct in general.

\begin{lem} \label{lem:eps-isometry} 
Let $V \subset \bbR^d$ be compact and such that $\theta(V) \ge \eta \diam(V)$ for some $\eta > 0$.  There is a constant $C$ depending only on $d$ such that, if $f : V \mapsto \bbR^d$ is an {\em $\eps$-isometry}, then there is an isometry $R : \bbR^d \mapsto \bbR^d$ such that $\max_{x \in V} \|f(x) - R(x)\| \le C \eps/\eta$. 
\end{lem}
 
\begin{proof}
This is a direct consequence of \citep[Th~3.3]{alestalo2001isometric}.
\end{proof}

\begin{lem} \label{lem:affine-near-sim}
Let $T : \bbR^d \mapsto \bbR^d$ be an affine function that transforms a regular simplex of edge length 1 into an $\eta$-approximate regular simplex of maximum edge length $\lambda > 0$.
There is a constant $C$, depending only on $d$, and an isometry $R$, such that $\|T(x) - \lambda R(x)\| \le C \lambda \eta$ for all $x \in B(0,1)$.
\end{lem}

\begin{proof}
By invariance, we may assume $T$ is linear and that the regular simplex is formed by $0, z_1, \dots, z_d$ and has edge length 1.  
Letting $w_i = T(z_i)$, we have that $0, w_1, \dots, w_d$ form an $\eta$-approximate regular simplex of maximum edge length $\lambda := \max_i \|w_i\|$.   
\lemref{approx-simplex} gives $0, w'_1, \dots, w'_d$ forming a regular simplex of edge length $\lambda$ such that $\max_i \|w_i - w'_i\| \le C_1 \lambda \eta$ for some constant $C_1$.  Let $R$ be the orthogonal transformation such that $R(z_i) = w'_i/\lambda$ for all $i \in [d]$.  
We have $\|T (z_i) - \lambda R (z_i) \| = \|w_i - w'_i\| \le C_1 \lambda \eta$ for all $i$.  In matrix notation, letting $Z := [z_1 \dots z_d]$, we have 
\[\|T Z - \lambda R Z\| \le \|T Z - \lambda R Z\|_F = \sqrt{\sum_{i=1}^d \|T z_i - \lambda R z_i \|^2} \le \sqrt{d} \max_{i \in [d]} \|T z_i - \lambda R z_i \| \le \sqrt{d} \, C_1 \lambda \eta.\] 
At the same time, $\|T Z - \lambda R Z\| \ge \|T - \lambda R\|/\|Z^{-1}\|$ with $\|Z^{-1}\| =1/\sigma_d(Z) = 1/\sigma_d([0 z_1 \cdots z_d])$ being a positive constant depending only on $d$.  Hence, $\|T - \lambda R\| \le (\sqrt{d}/\sigma_d(Z)) C_1 \lambda \eta =: C_2 \lambda \eta$.
\end{proof}

\begin{lem} \label{lem:affine-close}
Suppose that $S_1, S_2 : \bbR^d \mapsto \bbR^d$ are two affinities such that $\max_{x \in B(y,r)} \|S_1(x) - S_2(x)\| \le \eta$ for some $y \in \bbR^d$ and $r > 0$.  Then $\|S_1(x) - S_2(x)\| \le 2 \eta \|x-y\|/r + \eta$ for all $x \in \bbR^d$. 
\end{lem}

\begin{proof}
By translation and scale invariance, assume that $y = 0$ and $r =1$.
Let $L_i = S_i - S_i(0)$.  For $x \in B(0,1)$, we $\|L_1(x) - L_2(x)\| \le \|S_1(x) - S_2(x)\| + \|S_1(0) - S_2(0)\| \le 2\eta$.  Hence, for $x \in \bbR^d$, $\|L_1(x) - L_2(x)\| \le 2\eta \|x\|$, which in turn implies that $\|S_1(x) - S_2(x)\| \le \|L_1(x) - L_2(x)\| + \|S_1(0) - S_2(0)\| \le 2\eta \|x\| + \eta$.
\end{proof}

\subsection{Proof of \thmref{rates}} \label{sec:proof-rates}


Without loss of generality, we may assume that $D_n := \diam(\phi_n(\Omega_n)) \ge 1$.  Indeed, suppose that $D_n < 1$, but different from 0, for otherwise $\phi_n$ is a degenerate similarity and the result follows.  Let $\tilde\phi_n = D_n^{-1} \phi_n$, which is isotonic on $\Omega_n$ and satisfies $\diam(\tilde\phi_n(\Omega_n)) = 1$.  If the result is true for $\tilde\phi_n$, there is a similarity $\tilde S_n$ such that $\max_{x \in \Omega_n} |\tilde\phi_n(x) - \tilde S_n(x)| \le C \eps _n$ for some constant $C$.  
(We implicitly assume that the set $\phi_n(\Omega_n)$ contains the origin, so that $\tilde\phi_n(\Omega_n)$ remains bounded.)
We then have $\max_{x \in \Omega_n} |\phi_n(x) - S_n(x)| \le C D_n \eps _n \le C \eps_n$, where $S_n := D_n \tilde S_n$ is also a similarity.  

Let $r = \rho(U)$, so that there is some $u_\star$ such that $B(u_\star, r) \subset U$.  
Let $\Lambda_n = \Omega_n \cap B(u_\star, r/2)$ and $\delta_n = \diam(\phi_n(\Lambda_n))$.
Let $w$ be any unit-norm vector and define $y_\pm = u_\star \pm (r/2-\eps_n) w$.  Let $x_\pm \in \Omega_n$ be such that $\|x_\pm - y_\pm\| \le \eps_n$.  Necessarily, $x_\pm \in \Lambda_n$ because the distance from $y_\pm$ to $\partial B(u_\star, r/2)$ exceeds $\eps_n$.  Note that $\|x_- - x_+\| \ge r_1 := r - 4 \eps_n$.  
By isotonicity, 
\beq \label{iso1}
\|\phi_n(x) -\phi_n(x')\| \le \|\phi_n(x_-) - \phi_n(x_+)\| \le \delta_n, \text{ whenever $\|x-x'\| < r_1$.}
\eeq
Let $y_1, \dots, y_K$ be a $(r_1/3)$-packing of $U$, so that $K \le C (\diam(U)/r)^d$ for some constant $C>0$.  Let $x_{i_k} \in \Omega_n$ be such that $\|x_{i_k} - y_k\| \le \eps_n$, so that $U \subset \bigcup_{k \in [K]} B(y_k, r_1/3) \subset \bigcup_{k \in [K]} B(x_{i_k}, r_2)$, where $r_2 := r_1/3 + \eps_n$.  
Let $z_k = x_{i_k}$ for clarity.
Take $x, x' \in \Omega_n$.  Because $U$ is open, it is path-connected, so there is a continuous curve $\gamma : [0,1] \mapsto U$ such that $\gamma(0) = x$ and $\gamma(1) = x'$.  Let $k_0 \in [K]$ be such that $x \in B(z_{k_0}, r_2)$ and $s_0 = 0$.  Then for $j \ge 0$, let $s_{j+1} = \inf\{s > s_j : \gamma(s) \notin \bigcup_{l \in [j]} B(z_{k_l}, r_2)\}$, and let $k_{j+1} \in [K]$ be such that $\|z_{k_{j+1}} - \gamma(s_{j+1})\| \le \eps_n$.  Let $J = \min\{j : s_{j+1} = \infty\}$, which is indeed finite.  
By construction, $\|z_{k_j} -z_{k_{j+1}}\| \le 2 r_2 < r_1$ when $\eps_n < r/10$. 
By \eqref{iso1}, we have $\|\phi_n(z_{k_j}) - \phi_n(z_{k_{j+1}})\| \le \delta_n$.  Thus, by the triangle inequality, $\|\phi_n(x) - \phi_n(x')\| \le J \delta_n \le K \delta_n$.
This being true for all $x,x' \in \Omega_n$, this prove that $\delta_n \ge D_n/K \propto D_n (\diam(U)/r)^{-d}$.

\def\F{\hat\phi_n}
\medskip\noindent {\bf 1-NN interpolation.}
Let $\hat\phi_n$ denote the 1-NN interpolation of $\phi_n$ as in \eqref{1NN}. 
We claim that there is a $C^\star_0 \propto D_n/r$ and $c^\star_0 \propto (\diam(U)/r)^{-d} D_n/r$ such that $\hat\phi_n$ satisfies the following properties: for all $y,y',y^\dag,y^\ddag \in U$,
\begin{gather}
\|\hat\phi_n(y) - \hat\phi_n(y')\| \le C_0^\star (\|y - y'\| + \eps_n), \label{hat-Lip} \\
\|y - y'\| < \|y^\dag - y^\ddag\| - 4\eps_n \implies \|\hat\phi_n(y) - \hat\phi_n(y')\| \le \|\hat\phi_n(y^\dag) - \hat\phi_n(y^\ddag)\| + C_0^\star \eps_n, \label{hat-isotonic}
\end{gather}
and also
\beq\begin{array}{c}
\|\hat\phi_n(y) - \hat\phi_n(y')\| \ge c_0^\star \|y - y'\| - C_0^\star \eps_n, \\
\text{if $y,y' \in B(u_\star, r/2)$ satisfy $\|y-y'\| \ge 10 \eps_n$,} \label{hat-diam-bound} \\
\end{array}\eeq
and
\beq\begin{array}{c}
\|y - y'\| = \|y^\dag - y^\ddag\| \pm \eta \implies \|\hat\phi_n(y) - \hat\phi_n(y')\| = \|\hat\phi_n(y^\dag) - \hat\phi_n(y^\ddag)\| \pm C_0^\star(\eta + \eps_n), \\
\text{if $y,y' \in B(u_\star, r/2)$, $\eps_n < r/120$ and $0 \le \eta \le r/5$.} \label{hat-near-sim}
\end{array}\eeq
Indeed, let $x,x',x^\dag,x^\ddag \in \Omega_n$ such that $\|x - y\|, \|x' -y'\|, \|x^\dag - y^\dag\|, \|x^\ddag - y^\ddag\| \le \eps_n$.

For \eqref{hat-Lip}, we start by applying \lemref{1nn} to get 
\[\begin{split}
\|\hat\phi_n(y) - \hat\phi_n(y')\| 
&= \|\phi_n(x) - \phi_n(x')\| \pm 2 \omega_n(2 \eps_n) \\
&\le \omega_n(\|x-x'\|) + 2 \omega_n(2 \eps_n) \le \omega_n(\|y-y'\| + 2\eps_n) + 2 \omega_n(2 \eps_n),
\end{split}\] 
where $\omega_n$ is the modulus of continuity of $\phi_n$.  We then use \lemref{Lip}, which gives that $\omega_n(\eta) \le C \eta$ for all $\eta$ and some $C \propto D_n/r$, to get $\omega_n(\|y-y'\| + 2\eps_n) \pm 2 \omega_n(2 \eps_n) \le C(\|y-y'\| + 6\eps_n)$.

For \eqref{hat-isotonic}, we first note that $\|x-x'\| < \|x^\dag - x^\ddag\|$ by the triangle inequality, which in turn implies that $\|\phi_n(x) - \phi_n(x')\| \le \|\phi_n(x^\dag) - \phi_n(x^\ddag)\|$ since $\phi_n$ is isotonic.  We then apply \lemref{1nn} to get that $\|\hat\phi_n(y) - \hat\phi_n(y')\| \le \|\hat\phi_n(y^\dag) - \hat\phi_n(y^\ddag)\| + 4 \omega_n(2 \eps_n)$, and conclude with \lemref{Lip} as for \eqref{hat-Lip}.

For \eqref{hat-diam-bound}, we may apply \lemref{diam-bound} with $\Lambda_n$.  Let $V$ be the convex hull of $\Lambda_n$, so that $V \subset B(u_\star, r/2)$.  Let $z$ be a point in that ball.  If $z \ne u_\star$, let $w = (u_\star - z)/\|u_\star - z\|$, and if $z = u_\star$, let $w$ be any unit-norm vector.  Define $z' = z + \eps_n w$ and notice that the distance from $z'$ to $\partial B(u_\star, r/2)$ exceeds $\eps_n$.  Therefore, if $x \in \Omega_n$ is such that $\|z'-x\| \le \eps_n$, then necessarily, $x \in \Lambda_n$.  We then note that $\|z-x\| \le 2\eps_n$.  We conclude that $\delta_H(\Lambda_n, V) \le 2\eps_n$.  
Since $\|x -x'\| \ge \|y-y'\| - 2\eps_n \ge 4 (2\eps_n)$, we get that $\|\phi_n(x) - \phi_n(x')\| \ge c \|x-x'\|$, with $c := \diam(\phi_n(\Lambda_n))/5\diam(\Lambda_n) \ge \delta_n/5r$.  We then apply \lemref{1nn} to obtain $\|\hat\phi_n(y) - \hat\phi_n(y')\| \ge c \|x-x'\| - 2\omega_n(2\eps_n) \ge c \|y-y'\| - 2 (c + C) \eps_n$, using \lemref{Lip} as for \eqref{hat-Lip}.

For \eqref{hat-near-sim}, note that $x,x' \in B(u_\star,r/2+\eps_n) \subset B(u_\star, 3r/4)$, and $\|x - x'\| = \|x^\dag - x^\ddag\| \pm (\eta + 4 \eps_n)$ by the triangle inequality.  By \lemref{near-sim} --- where the constant there is denoted here by $C' \propto D_n/r$ --- this implies that 
\[\|\phi_n(x) - \phi_n(x')\| = \|\phi_n(x^\dag) - \phi_n(x^\ddag)\| \pm C'(\eta + \eps_n)\]
when $\eta + 4 \eps_n < r/4 - 2\eps_n$, which is true when $\eps_n < r/120$ and $\eta \le r/5$.
We then apply \lemref{1nn} together with \lemref{Lip}, as for \eqref{hat-Lip}.

\medskip\noindent {CASE $d = 1$.}
This case is particularly simple. 
Note that $U$ is a bounded open interval of $\bbR$. 
We show that the function $\F$ is approximately midlinear on $U$.  
Take $x,y \in U$ and define $\mu = (x+y)/2$.  By the fact that $\F$ takes its values in $\bbR$, and \eqref{hat-near-sim}, we have
\[
\Big|\tfrac12 (\F(x) + \F(y)) - \F(\mu)\Big| = \tfrac12 \Big| |\F(x) - \F(\mu)| - |\F(y) - \F(\mu)| \Big| \le C^\star_0 \eps_n/2,
\]
when $\eps_n/r$ is small enough.
Hence, $\F$ is $(C^\star_0 \eps_n)$-approximate midlinear on $U$. 
By the result of \cite{vestfrid2003linear}, namely \lemref{near-midpoint}, there is $C \propto 1$ --- since $U$ is a ball --- and an affine function $T_n$ such that $\max_{y\in U} |\F(y) - T_n(y)| \le C C^\star_0 \eps_n$.  
Since all affine transformations from $\bbR$ to $\bbR$ are (possibly degenerate) similarities, we conclude.

\medskip\noindent {CASE $d \ge 2$.}
For the remaining of this subsection, we assume that $d \ge 2$.

\medskip\noindent {\bf Approximate midlinearity.}
We show that there is a constant $C$ such that $\F$ is locally $C \eps_n$-approximately midlinear.
Take $x,y \in B(u_\star, r/4)$, and let $\mu = (x+y)/2$.  
Let $t > 0$ be a constant to be set large enough later. 

If $\|x - y\| \le t \eps_n$, then by \eqref{hat-Lip}, $\F(x), \F(y) \in B(\F(\mu), C_0^\star(t/2+1) \eps_n)$, so that 
\[\big\|\F(\mu) -\frac12 (\F(x) + \F(y))\big\| \le C_0^\star(t/2+1) \eps_n.\]  

Therefore, assume that $\|x - y\| \ge t \eps_n$.
Let $z_1, \dots, z_d$ be constructed as in the proof of \thmref{infinite}.  By construction, both $x, z_1, \dots, z_d$ and $y, z_1, \dots, z_d$ form regular simplexes, and $\mu$ is  the barycenter of $z_1, \dots, z_d$.
By \lemref{simplex}, for any $i \ne j$, 
\[\|z_i - \mu\| = \sqrt{(d-1)/2d} \, \|z_i - z_j\| = \sqrt{(d-1)/2d} \, \sqrt{2d/(d+1)} \, \|x-\mu\| \le \|x-y\|/2,\]
which coupled with the fact that $x,y \in B(u_\star, r/4)$ yields that $z_i \in B(u_\star, r/2)$ for all $i$.
Now, let $z_0 = x$.  By \eqref{hat-near-sim}, we have $\min_{i \ne j} \|\F(z_i) - \F(z_j)\| \ge \max_{i,j} \|\F(z_i) - \F(z_j)\| - C_0^\star \eps_n$.
Let $c_d = \sqrt{d/(2d+2)}$.
By \eqref{hat-diam-bound} and \lemref{simplex}, 
\[
\|\F(z_i) - \F(z_j)\| \ge c_0^\star \|z_i - z_j\| - C_0^\star \eps_n = c_0^\star c_d \|x - y\| - C_0^\star \eps_n \ge \big(c_0^\star c_d t - C_0^\star \big) \eps_n.
\]
Hence, assuming $t \ge 2 C_0^\star/c_0^\star c_d$, we have $\min_{i \ne j} \|\F(z_i) - \F(z_j)\| \ge (1 - \eta) \max_{i,j} \|\F(z_i) - \F(z_j)\|$, where $\eta := 2C_0^\star/(c_0^\star c_d t)$.  In that case, $\hat\phi_n(x), \hat\phi_n(z_1), \dots, \hat\phi_n(z_d)$ form a $\eta$-approximate regular simplex.  By symmetry, the same is true of $\hat\phi_n(y), \hat\phi_n(z_1), \dots, \hat\phi_n(z_d)$.

Define $\lambda = \|\F(x) - \F(y)\|$.
By \lemref{simplex}, $\|z_i - z_j\| = c_d \|x - y\| < \|x -y\| - 4\eps_n$ when $t > 4/(1-c_d)$, since $\|x - y\| \ge t \eps_n$ and $c_d < 1$.
By \eqref{hat-isotonic}, this implies that $\|\F(z_i) -\F(z_j)\| \le \lambda + C_0^\star \eps_n$.   By \eqref{hat-diam-bound}, $\lambda \ge (c_0^\star t  -C_0^\star) \eps_n$, so that $\lambda + C_0^\star \eps_n \le 2\lambda$ since we already assumed that $t \ge 2 C_0^\star/c_0^\star c_d > 2 C_0^\star/c_0^\star$.

For $a \in \{x,y,\mu\}$, $\|a - z_i\|$ is constant in $i \in [d]$. 
Therefore, by \eqref{hat-near-sim}, $\min_i \|\F(a) - \F(z_i)\| \ge \max_i \|\F(a) - \F(z_i)\| - C_0^\star \eps_n$.  Define $\xi_a$ as the orthogonal projection of $\F(a)$ onto the affine space $A := \aff(\F(z_1), \dots, \F(z_d))$ and let $\delta_a = \|\F(a) - \xi_a\|$.  By the Pythagoras theorem, we have $\|\xi_a - \F(z_i)\|^2 = \|\F(a) - \F(z_i)\|^2 - \delta_a^2$.  In particular, 
\[\begin{split}
\max_i \|\xi_a - \F(z_i)\|^2 - \min_i \|\xi_a - \F(z_i)\|^2 
&= \max_i \|\F(a) - \F(z_i)\|^2 - \min_i \|\F(a) - \F(z_i)\|^2 \\
&\le 2C_0^\star \eps_n \min_i \|\F(a) - \F(z_i)\| + (C_0^\star \eps_n)^2 \le C_1 \eps_n,
\end{split}\] 
where $C_1 := 2C_0^\star D_n + {C_0^\star} r$, once $\eps_n \le r$. 
Let $\zeta$ denotes the barycenter of $\F(z_1), \dots, \F(z_d)$.
Assume that $t$ is sufficiently large that $\eta \le 1/C_2$, where $C_2 \propto 1$ is the constant of \lemref{barycenter}.  By that lemma, and the fact that $\F(z_1), \dots, \F(z_d)$ form a $\eta$-approximate regular simplex of maximum edge length bounded by $\lambda$, we have $\|\xi_a - \zeta\| \le C_2 \lambda C_1 \eps_n$.
Let $L$ be the line passing through $\zeta$ and perpendicular to $A$.  
We just proved that $\F(x), \F(y), \F(\mu)$ are within distance $C_3 \lambda \eps_n$ from $L$, where $C_3 := C_1C_2$.  

Let $\xi$ denote the orthogonal projection of $\F(\mu)$ onto $(\F(x)\F(y))$.
Since $\|x - \mu\| = \|y - \mu\|$, we can apply \eqref{hat-near-sim} to get
\[\begin{split}
& \big|\|\xi -\F(x)\|^2 - \|\xi - \F(y)\|^2\big| \\
&= \big|\|\F(\mu) - \F(x)\|^2 - \|\F(\mu) - \F(y)\|^2\big| \\
&= \big|\|\F(\mu) - \F(x)\| + \|\F(\mu) - \F(y)\|\big| \times \big|\|\F(\mu) - \F(x)\| - \|\F(\mu) - \F(y)\|\big| \\
&\le 4 C^\star_0 \lambda \eps_n,
\end{split}\]
using the fact that $\max(\|\F(\mu) - \F(x)\|, \|\F(\mu) - \F(y)\|) \le \lambda + C^\star_0 \eps_n \le 2 \lambda$, due to \eqref{hat-isotonic} and $\|x - \mu\| = \|y - \mu\| = \frac12 \|x-y\| < \|x-y\| - 4\eps_n$ when $t$ is large enough.  
By \lemref{barycenter}, we then obtain $\|\xi - \frac12 (\F(x) + \F(y)) \| \le C_4 \lambda \eps_n$ for some constant $C_4 \propto C_0^\star$.  
In particular, recalling that $\lambda = \|\F(x) - \F(y)\|$, this implies that $\xi \in [\F(x)\F(y)]$ when $\eps_n \le 1/2C_4$.
    
It remains to argue that $\F(\mu)$ is close to $\xi$.    
We already know that $\F(x), \F(y), \F(\mu)$ are within distance $C_3 \lambda \eps_n$ from $L$, and by convexity, the same must be true of $\xi$.  Let $M = (\F(x)\F(y))$ and $\theta = \angle (L, M)$.  
Let $P_M$ denote the orthogonal projection onto $M$, when $M$ is a linear subspace.
By Pythagoras theorem, 
\[\begin{split}
\lambda^2 = \|\F(x) - \F(y)\|^2 
&= \|P_L (\F(x) - \F(y))\|^2 + \|P_{L^\perp} (\F(x) - \F(y))\|^2 \\
&\le (\cos \theta)^2 \lambda^2 + (2 C_3 \lambda \eps_n)^2,
\end{split}\]
implying that $\sin \theta \le 2 C_2 \eps_n$.
Since $\|P_L - P_M\| = \sin \theta$ and $\F(\mu) - \xi$ is parallel to $M$, we also have 
\[\begin{split}
\|\F(\mu) - \xi\|^2  
&= \|P_L (\F(\mu) - \xi)\|^2 + \|P_{L^\perp} (\F(\mu) - \xi)\|^2 \\
&\le (\sin \theta)^2 \|\F(\mu) - \xi\|^2 + (2 C_3 \lambda \eps_n)^2,
\end{split}\]
so that $\|\F(\mu) - \xi\| \le 2 C_3 \lambda \eps_n / \cos \theta \le 2 C_3 \lambda \eps_n / \sqrt{1 - (2C_3 \eps_n)^2} \le C_5 \lambda \eps_n,$ 
for some constant $C_5 \propto C_3$, once $C_3 \eps_n$ is small enough.
  
We conclude that $\|\F(\mu) - \tfrac12 (\F(x) + \F(y)) \| \le (C_4 +C_5) \lambda \eps_n$, by the triangle inequality.

\medskip\noindent {\bf Approximate affinity.}
We now know that $\F$ is $C \eps_n$-approximate midlinear on $B(u_\star, r/4)$ for some constant $C \propto C_0^\star (D_n + r) \propto C_0^\star (\diam(Q) + r)$.  This implies, by the result of \cite{vestfrid2003linear}, that is \lemref{near-midpoint}, that there is an affine function $T_n$ such $\|\F(x) - T_n(x)\| \le C_1^\star \eps_n$ for all $x \in W$, for some constant $C_1^\star \propto r C \propto r C_0^\star (\diam(Q) + r)$.

\medskip\noindent {\bf Approximate similarity.}
(Reinitialize the constants $C_k, k \ge 1$.)
We saw above that $\F$ transforms the regular simplex $z_0 (= x), z_1, \dots, z_d$ with height denoted $h$ satisfying $h \ge t \eps_n/2$ into a $\eta$-approximate one, where $\eta = 2C_0^\star/(c_0^\star c_d t)$.  
In what follows, choose these points so that they are all in $B(u_\star, r/2)$ and the simplex has height $h \ge r/8$.  
(From here on, reinitialize the variables $x, y, \lambda$, etc.)
We can then take $t = r/4\eps_n$, yielding $\eta = C_1 \eps_n$ for a constant $C_1 \propto C_0^\star/(c_0^\star r)$.
By the triangle inequality, we have
\[\begin{split}
\min_{i \ne j} \|T_n(z_i) - T_n(z_j)\| 
&\ge \min_{i \ne j} \|\F(z_i) - \F(z_j)\| - 2 C_1^\star \eps_n \\
&\ge (1 - C_1 \eps_n) \max_{i \ne j} \|\F(z_i) - \F(z_j)\| - 2 C_1^\star \eps_n \\
&\ge \max_{i \ne j} \|T_n(z_i) - T_n (z_j)\| - (4 C_1^\star + C_1 \delta_n) \eps_n.
\end{split}\]
By the triangle inequality and \eqref{hat-diam-bound}, 
\[\begin{split}
\gamma_n := \max_{i,j} \|T_n(z_i) - T_n(z_j)\| 
&\ge \max_{i,j} \|\F(z_i) - \F(z_j)\| - 2 C_1^\star \eps_n \\
& \ge c_0^\star \max_{i,j} \|z_i - z_j\| - C_0^\star \eps_n - 2 C_1^\star \eps_n 
\ge c_0^\star r/8 - (C_0^\star + 2C_1^\star) \eps_n.
\end{split}\]
Hence, we find that $T_n(z_0), \dots, T_n(z_d)$ form a $C_2 \eps_n$-approximate regular simplex, where $C_2 := (4 C_1^\star + C_1 \delta_n)/(c_0^\star r/8 - (C_0^\star + 2C_1^\star) \eps_n)$.
Note that its maximum edge length is bounded as follows:
\[
\gamma_n \le \max_{i,j} \|\F(z_i) - \F(z_j)\| + 2 C_1^\star \eps_n \le \delta_n + 2 C_1^\star \eps_n \le 2 \delta_n,
\]
when $2C_1^\star \eps_n \le \delta_n$.
By \lemref{affine-near-sim}, there is a constant $C_3 > 0$ and an isometry $R^\star_n$, such that we have $\max_{x \in W} \|T_n(x) - \lambda_n R^\star_n(x)\| \le C_3 \lambda_n C_2 \eps_n$, where $\lambda_n := \gamma_n/h$.  
Because $r/8 \le h \le r$ and the bounds on $\gamma_n$ above, there is a constant $C^\star_2 \ge 1$ such that 
\beq \label{lambda}
1/C^\star_2 \le \lambda_n \le C^\star_2.
\eeq
This implies that 
\beq \label{near-sim-W}
\|\F(x) - \lambda_n R^\star_n(x)\| \le \|\F(x) - T_n(x)\| + \|T_n(x) - \lambda_n R^\star_n(x)\| \le (C^\star_1 + C_3 C_2 C^\star_2) \eps_n =: C^\star_3 \eps_n.
\eeq

\medskip\noindent {\bf Covering and conclusion.}
\def\G{{\tilde \phi_n}}
(Reinitialize the constants $C_k, k \ge 1$.)
Let $u_1 = u_\star$ and let $u_2, \dots, u_K \in U$ be such that $u_1, \dots, u_K$ form a maximal $(r/16)$-packing of $U$.  
(The number 16 is not essential here, but will play a role in the proof of \thmref{local-quad}.)
Note that $U = U_1 \cup \cdots \cup U_K$ where $U_k := U \cap B(u_k, r/4)$, and note that $U_\star := U_1 \subset U$.
For $u, u' \in U_k$, there are $w,w' \in U_\star$ such that $\|w-w'\| = \|u-u'\|$.  Define $\G = \F/\lambda_n$.  By \eqref{hat-near-sim}, and then \eqref{lambda}-\eqref{near-sim-W}, we have 
\[\begin{split}
\|\G(u) - \G(u')\| 
&= \|\G(w) - \G(w')\| \pm C^\star_0 \eps_n/\lambda_n \\
&= \|w-w'\| \pm (C^\star_0 + C^\star_3) \eps_n/C^\star_2 =: \|w -w'\| \pm C_1 \eps_n.
\end{split}\]
Let 
\beq \label{xi1}
\xi_1 = \min_k \frac{\theta(U_k)}{\diam(U_k)},
\eeq 
which is strictly positive.
The result of \cite{alestalo2001isometric}, namely \lemref{eps-isometry}, gives a constant $C_2 \propto \xi_1$ and an isometry $R_k$ such that $\max_{u \in U_k} \|\G(u) - R_k(u)\| \le C_2 \eps_n$.  

Let 
\beq \label{xi2}
\xi_2 = \frac12 \min \big\{\rho(U_k \cap U_{k'}) : U_k \cap U_{k'} \ne \emptyset\big\}.
\eeq 
Take $k, k' \in [K]$ such that $U_k \cap U_{k'} \ne \emptyset$, so that there is $u \in U$ such that $B(u, \xi_2) \subset U_k \cap U_{k'}$.  Since 
\[\begin{split}
\max_{x \in B(u, \xi_2)} \|R_k(x) - R_{k'}(x)\| 
&\le \max_{x \in U_k \cap U_k'} \|R_k(x) - \G(x)\| + \|\G(x) - R_{k'}(x)\|
\\&\le \max_{x \in U_k} \|R_k(x) - \G(x)\| + \max_{x \in U_{k'}}\|\G(x) - R_{k'}(x)\|
\le 2 C_2 \eps_n,
\end{split}\]
we have $\|R_k(x) - R_{k'}(x)\| \le (2 \|x - u\|/\xi_2 + 1) 2 C_2 \eps_n$ for all $x \in \bbR^d$, by \lemref{affine-close}.  
Hence, $\|R_k(x) - R_{k'}(x)\| \le (2 \diam(U)/\xi_2 + 1) 2 C_2 \eps_n =: C_3 \eps_n$ for all $x \in U$.
If instead $U_k \cap U_{k'} = \emptyset$, we do as follows.  Since $U$ is connected, there is a sequence $k_0 = k, k_1, \dots, k_m = k'$ in $[K]$, such that $U_{k_i} \cap U_{k_{i+1}} \ne \emptyset$.  We thus have $\max_{x \in U} \|R_{k_i}(x) - R_{k_{i+1}}(x)\| \le C_3 \eps_n$.  By the triangle inequality, we conclude that $\max_{x \in U} \|R_k(x) - R_{k'}(x)\| \le K C_3 \eps_n$ for any $k,k' \in [K]$.  
Noting that $R_1 = R_n^\star$ (since $U_1 = U_\star$), for any $k \in [K]$ and $x \in U_k$,
\[
\|\G(x) - R^\star_n(x)\| \le \|R_k(x) - R_1(x)\| + C_2 \eps_n \le (K C_3 + C_2) \eps_n.
\] 
We conclude that, for any $x \in U$,
\beq \label{last-bound}
\|\F(x) - \lambda_n R_n^\star(x)\| \le (K C_3 + C_2) \lambda_n \eps_n \le (K C_3 + C_2) C^\star_2 \eps_n =: C_4 \eps_n.
\eeq
This concludes the proof when $d \ge 2$.

\medskip\noindent {\bf A refinement of the constant.}
Assume now that $U = U^h$ for some $h > 0$.  
Tracking the constants above, we see that they all depend only on $(d, \rho(U), \diam(U), \diam(Q))$, as well as $\xi_1$ and $\xi_2$ defined in \eqref{xi1} and \eqref{xi2}, respectively.
We note that $\diam(U_k) \le r$ and $\rho(U_k) \ge \min(r/2,h)$ by \lemref{inter-vol}, so that $\xi_1 \ge \min(r/2,h)/r$.  To bound $\xi_2$, we can do as we did at the beginning of this section, so that at the end of that section, we can restrict our attention to chains $k_0, \dots, k_m$ where $\|u_{k_j} - u_{k_{j+1}}\| \le 2 r/16 = r/8$.  To be sure, fix $k,k' \in [K]$ and let $\gamma : [0,1] \mapsto U$ be a curve such that $\gamma(0) = u_k$ and $\gamma(1) = u_{k'}$.  Define $s_0 = 0$ and then $s_{j+1} = \inf\{s > s_j : \|\gamma(s) - u_{k_j}\| > r/16\}$, and let $k_{j+1} \in [K]$ be such that $\|\gamma(s_{j+1}) - u_{k_{j+1}}\| \le r/16$, which is well-defined since $(u_k, k \in [K])$ is a $(r/16)$-packing of $U$.  We then have
\[
\|u_{k_j} - u_{k_{j+1}}\| \le \|u_{k_j} - \gamma(s_{j})\| - \|\gamma(s_{j+1}) - u_{k_{j+1}}\| \le r/16 + r/16 = r/8.
\]
We can therefore redefine $\xi_2$ in \eqref{xi2} as $\frac12 \min\{\rho(U_k \cap U_{k'}) : \|u_k - u_{k'}\| \le r/8\}$.
Because $U = U^h$, for each $k \in [K]$, there is $v_k$ such that $u_k \in B(v_k, \min(r/16,h)) \subset U$.  By the triangle inequality, $B(v_k, \min(r/16,h)) \subset U_{k'}$ when $\|u_k - u_{k'}\| \le r/8$, so that $\xi_2 \ge \min(r/16, h)$.
So we see that everything depends on $(d, h, \rho(U), \diam(U), \diam(Q))$.  The second part of the theorem now follows by invariance considerations.


\subsection{Proof of \lemref{graph}}

Let $c = \essinf_U f$ and $C = \esssup_U f$, which by assumption belong to $(0,\infty)$.
Fix $i \in [n]$ and let $N_i = \# \{j \ne i : \|x_j -x_i\| \le r\}$.  For $j \ne i$, $p_{i}(j) := \P(\|x_j -x_i\| \le r) = \int_{B(x_i, r)} f(u) {\rm d}u$.
For an upper bound, we have
\[
p_{i}(j) \le C \vol(B(x_i, r) \cap U) \le C \vol(B(x_i, r)) = C \zeta_d r^d =: Q,
\]
where $\vol$ denotes the Lebesgue measure in $\bbR^d$ and $\zeta_d$ is the volume of the unit ball in $\bbR^d$.  
Hence, $\P(N_i > 2 (n-1) Q) \le \P(\Bin(n-1, Q) > 2 (n-1) Q) \le e^{- (n-1) Q/3}$ by Bennett's inequality for the binomial distribution.  By the union bound, we conclude that $\max_i N_i \le 2 (n-1) Q$ with probability at least $1 - n e^{- (n-1) Q/3}$, which tends to 1 if $n r^d \ge C_0 \log n$ and $C_0 > 0$ is sufficiently large.

For a lower bound, we use the following lemma.
\begin{lem} \label{lem:inter-vol}
Suppose $U \subset \bbR^d$ is open and such that $U = U^h$ for some $h > 0$.  Then for any $x \in U$ and any $r > 0$, $B(x, r) \cap U$ contains a ball of radius $\min(r,h)/2$.  Moreover, the closure of that ball contains $x$.
\end{lem}

\begin{proof}
By definition, there is $y \in U$ such that $x \in B(y,h) \subset U$.  We then have $B(x, r) \cap U \supset B(x,r) \cap B(y,h)$, so it suffices to show that the latter contains a ball of radius $\min(r,h)/2$.  By symmetry, we may assume that $r \le h$.  
If $\|x - y\| \le r/2$, then $B(x, r/2) \subset B(y,h)$ and we are done.
Otherwise, let $z = (1-t) x + t y$ with $t := r/2 \|x - y\| \in (0,1)$, and note that $B(z, r/2) \subset B(x,r) \cap B(y,h)$ and $x \in \partial B(z, r/2)$.
\end{proof}

Now that \lemref{inter-vol} is established, we apply it to get
\[
p_{i}(j) \ge c \vol(B(x_i, r) \cap U) \ge c \zeta_d (\min(r,h)/2)^d =: q.
\]
Hence, $\P(N_i < (n-1) q/2) \le \P(\Bin(n-1, q) < (n-1) q/2) \le e^{- (6/7)(n-1)q}$.  By the union bound, we conclude that $\min_i N_i \ge (n-1)q/2$ with probability at least $1 - n e^{- (6/7)(n-1)q}$, which tends to 1 if $n r^d \ge C_1 \log n$ and $C_1 > 0$ is sufficiently large.  (Recall that $h$ is fixed.)

\subsection{More auxiliary results}

We list here a few additional of auxiliary results that will be used in the proof of \thmref{local-quad}.

For $V \subset \bbR^d$ and $x,x' \in V$, define the intrinsic metric 
\[\delta_V(x,x') = \sup \Big\{L : \exists \gamma : [0,L] \mapsto V, \text{ 1-Lipschitz, with } \gamma(0) = x, \gamma(L) = x'\Big\},\]
where $\gamma$ is 1-Lipschitz if $\|\gamma(s) - \gamma(t)\| \le |s-t|$ for all $s,t \in [0,L]$.
If no such curve exists, set $\delta_V(x,x') = \infty$.  The intrinsic diameter of $V$ is defined as $\sup\{\delta_V(x,x') : x,x' \in V\}$.
We note that, if $L := \delta_V(x,x') < \infty$, then there is a curve $\gamma \subset \bar V$ with length $L$ joining $x$ and $x'$.
Recall that a curve with finite length is said to be rectifiable.
See \citep{burago2001course} for a detailed account of intrinsic metrics.

For $U \subset \bbR^d$ and $h > 0$, let $U^{\ominus h} = \{x \in U : B(x,h) \subset U\}$.  This is referred to as an erosion (of the set $U$) in mathematical morphology.

\begin{lem} \label{lem:open-intrinsic}
If $U \subset \bbR^d$ is open and connected, then for each pair of points $x,x' \in U$, there is $h > 0$ and a rectifiable curve within $U^{\ominus h}$ joining $x$ and $x'$.  
\end{lem}

\begin{proof}
Take $x,x' \in U$.  By taking an intersection with an open ball that contains $x,x'$, if needed, we may assume without loss of generality that $U$ is bounded.
Since every connected open set in a Euclidean space is also path-connected \citep[Example~2.5.13]{waldmann2014topology}, there is a continuous curve $\gamma : [0,1] \mapsto U$ such that $\gamma(0) = x$ and $\gamma(1) = x'$.  A priori, $\gamma$ could have infinite length.  However, $\gamma$ ($\equiv \gamma([0,1])$) is compact.  For each $t \in [0,1]$, let $r(t) > 0$ be such that $B_t := B(\gamma(t), r(t)) \subset U$.  Since $\gamma \subset \bigcup_{t \in [0,1]} B_t$, there is $0 \le t_1 < \dots < t_m \le 1$ such that $\gamma \subset \bigcup_{j \in [m]} B_{t_j}$.  Since $\gamma$ is connected, necessarily, for all $j \in [m-1]$ there is $s_j \in [t_j, t_{j+1}]$ such that $\gamma(s_j) \in B_{t_j} \cap B_{t_{j+1}}$.  Let $s_0 = 0$ and $s_m = 1$.  Then $[\gamma(s_j) \gamma(s_{j+1})] \subset B_{t_{j+1}} \subset U$ for all $j \in 0, \dots, m-1$, and therefore the polygonal line defined by $x = \gamma(s_0), \gamma(s_1), \dots, \gamma(s_{m-1}), \gamma(s_m) = x'$ is inside $\bigcup_{j \in [m]} B_{t_j} \subset U^{\ominus r}$ where $r := \min_{j\in [m]} r(t_j) > 0$.  By construction, this polygonal line joins $x$ and $x'$, and is also rectifiable since it has a finite number of vertices.
\end{proof}

\begin{lem} \label{lem:diam-finite}
Suppose $U \subset \bbR^d$ is bounded, connected, and such that $U = U^h$ for some $h > 0$.  Then there is $h_\ddag > 0$ such that, for all $h' \in [0, h_\ddag]$, the intrinsic diameter of $U^{\ominus h'}$ is finite.
\end{lem}

\begin{proof}
Let $V = U^{\ominus h}$.
By assumption, for all $x \in U$, there is $y \in V$ such that $x \in B(y,h) \subset U$.  In particular, $U \supset V \ne \emptyset$.  

Let $V_1$ be a connected component of $V$.  
Pick $y_1 \in V_1$ and note that $B_1 := B(y_1, h) \subset U$ by definition, and also $B_1 \subset V_1$ because $B_1$ is connected.  
Let $\zeta_d$ be the volume of the unit ball in $\bbR^d$.
Since the connected components are disjoint and each has volume at least $\zeta_d h^d$ while $U$ has volume at most $\zeta_d (\diam(U)/2)^d$, $V$ can have at most $\lceil (\diam(U)/2h)^d \rceil$ connected components, which we now denote by $V_1, \dots, V_K$.  
Pick $y_k \in V_k$ for each $k \in [K]$.
Applying \lemref{open-intrinsic}, for each pair of distinct $k, k' \in [K]$, there is a rectifiable (i.e., finite-length) path $\gamma_{k,k'} \subset U$ joining $y_k$ and $y_{k'}$.  By \lemref{open-intrinsic}, the length of $\gamma_{k,k'}$, denoted $D_{k,k'}$, is finite, and there is $h_{k,k'} > 0$ such that $\gamma_{k,k'} \subset U^{\ominus h_{k,k'}}$.  Let $D_\ddag = \max_{k,k' \in [K]} D_{k,k'}$ and $h_\ddag = \min_{k,k' \in [K]} h_{k,k'}$.

We now show that each connected component $V_k$ has finite diameter in the intrinsic metric of $V' := U^{\ominus h/2}$.  
Since $V_k$ is bounded, there is $x_1, \dots, x_m \in V_k$ such that $V_k \subset \bigcup_{j \in [m_k]} Q_j$, where $Q_j := B(x_j, h/2) \subset V'$.
Take any $x,x' \in V_k$.  Let $j,j' \in [m_k]$ be such that $x \in Q_j$ and $x' \in Q_{j'}$.
Since $V_k$ is connected, there is a sequence $j = j_0, j_1, \dots, j_{S_k} = j' \in [m_k]$ such that $Q_{j_s} \cap Q_{j_{s+1}} \ne \emptyset$ for all $s = 0, \dots, S_k$.  
Choose $z_s \in Q_{j_s} \cap Q_{j_{s+1}}$ and let $z_0 = x$ and $z_{S_k} = x'$.  Then $[z_s z_{s+1}] \subset Q_{j_{s+1}}$ for all $s$.  Let $L$ be the polygonal line formed by $z_0, \dots, z_{S_k}$.  By construction, $L \subset \bigcup_{s=0}^{S_k} Q_{j_s} \subset V'$, it joins $x$ and $x'$, and  has length at most $(S_k+1) 2h$.  Hence, $\delta_{V'}(x,x') \le (S_k+1) 2h \le 2 (m_k+1) h$.  This being valid for all $x,x' \in V_k$, we proved that $V_k$ has diameter at most $D_k := 2(m_k+1)h$ in the intrinsic metric of $V'$. 
Let $D_\star = \max_{k \in [K]} D_k$.

Now take $h_\dag \in [0,h_\ddag]$ and any $x,x' \in U^{\ominus h_\dag}$.
Let $y,y' \in V$ be such that $x \in B(y,h)$ and $x' \in B(y',h)$.
Let $k,k' \in [K]$ be such that $y \in V_k$ and $y' \in V_{k'}$.
There are curves $\gamma, \gamma' \subset V'$ of length at most $D_\star$ such that $\gamma$ joins $y$ and $y_k$, while $\gamma'$ joins $y'$ and $y_{k'}$.  We then join $y_k$ and $y_{k'}$ with $\gamma_{k,k'}$
All together, we have the curve $[xy] \cup \gamma \cup \gamma_{k,k'} \cup \gamma' \cup [y'x']$, which joins $x$ and $x'$, lies entirely in $U^{\ominus h_\dag}$, and has length bounded by $h + D_\star + D_\ddag + D_\star + h =: D$.
And this is true for any pair of such points. 
\end{proof}

\begin{lem} \label{lem:affine-close-discrete}
Suppose that $S_1, S_2 : \bbR^d \mapsto \bbR^d$ are two affinities such that $\max_j \|S_1(z_j) - S_2(z_j)\| \le \eps$, where $z_0, \dots, z_d$ form in a $\eta$-approximate regular simplex with minimum edge length at least $\lambda$.  There is $C>0$ depending only on $d$ such that, if $\eta \le 1/C$, then  
$\|S_1(x) - S_2(x)\| \le C \eps \|x - z_0\|/\lambda + \eps$ for all $x \in \bbR^d$. 
\end{lem}

\begin{proof}
Note that this is closely related to \lemref{affine-close}.
By translation and scale invariance, assume that $z_0 = 0$ and $\lambda = 1$.
Let $L_i = S_i - S_i(0)$.  We have $\|L_1(z_j) - L_2(z_j)\| \le \|S_1(z_j) - S_2(z_j)\| + \|S_1(0) - S_2(0)\| \le 2\eps$.  Let $Z$ denote the matrix with columns $z_1, \dots, z_d$.  In matrix notation, we have 
\[
\|(L_1 - L_2) Z\|_F = \sqrt{\sum_j \|(L_1 - L_2) z_j\|^2} \le 2 \sqrt{d} \eps.
\]
We also have $\|(L_1 - L_2) Z\|_F \ge \|(L_1 - L_2) Z\| \ge \sigma_d(Z) \|L_1 - L_2\|$, and by \lemref{simplex-value}, $\sigma_d(Z) = \sigma_d([z_0, Z]) \ge 1/C_1$ when $\eta \le 1/C_1$, where $C_1$ depends only on $d$.  
In that case, $\|L_1 - L_2\| \le C_2 \eps$ for another constant $C_2$.  Equivalently, for $x \in \bbR^d$, $\|L_1(x) - L_2(x)\| \le C_2 \eps \|x\|$, which in turn implies that $\|S_1(x) - S_2(x)\| \le \|L_1(x) - L_2(x)\| + \|S_1(0) - S_2(0)\| \le C_2 \eps \|x\| + \eps$.
\end{proof}

\subsection{Proof of \thmref{local-quad}}

Because $\phi_n$ is bounded independently of $n$, we may assume without loss of generality that $C_0 \eps_n \le r_n$ and $C_0 r_n \le h$ for all $n$, where $C_0 \ge 1$ will be chosen large enough later on.

Take $y \in U$ and let $\Omega_y = \Omega_n \cap B(y, r_n)$ and $Q_y = \phi_n(\Omega_y)$.
We first show that there is $C_1 \propto \diam(Q)/\rho(U)$ such that, for any $y \in U$, $\diam(Q_y) \le C_1 r_n$.
For this, we mimic the proof of \lemref{Lip}.  
Take $x,x' \in \Omega_y$ such that $\xi := \|\phi_n(x) - \phi_n(x')\| = \diam(Q_y)$.
Let $u$ be such that $B(u, \rho(U)) \subset U$.  Let $y_1, \dots, y_m$ be an $(r_n + 2 \eps_n)$-packing of $B(u, \rho(U))$ with $m \ge A_1 (\rho(U)/r_n)^{d}$ for some $A_1 \propto 1$.  Then let $\{x_{i_s} : s \in [m]\} \subset \Omega_n$ be such that $\max_{s \in[m]} \|y_s - x_{i_s}\| \le \eps_n$.  By the triangle inequality, for all $s \ne t$, we have $\|x_{i_s} - x_{i_t}\| \ge \|y_s - y_t\| - 2\eps_n \ge r_n > \|x - x'\|$.  By \eqref{outside}, we have $\|\phi_n(x_{i_s}) - \phi_n(x_{i_t})\| \ge \xi$, so that $\phi_n(x_{i_1}), \dots, \phi_n(x_{i_m})$ form a $\xi$-packing.
Therefore $m \le A_2 (\diam(Q)/\xi)^{d}$ for some $A_2 \propto 1$.  We conclude that $\xi \le (A_2/A_1)^{1/d} (\diam(Q)/\rho(U)) r_n =: C_1 r_n$.

We apply \thmref{rates} to $U_y := B(y,r_n)$ and $\Omega_y$.  With the fact that $\delta_H(\Omega_y, U_y) \le 2\eps_n$ --- as we saw in the proof of \eqref{hat-diam-bound} --- and invariance considerations, we obtain a constant $C \propto 1$ and a similarity $S_y$ such that $\max_{x \in \Omega_y} \|\phi_n(x) - S_y(x)\| \le C (\diam(Q_y)/r_n) \eps_n \le C C_1 \eps_n =: C_2 \eps_n$.
(Note that all the quantities with subscript $y$ depend also on $n$, but this will be left implicit.)

Fix $y_\star \in U^{\ominus r_n}$.
For $x \in \Omega_n$, there is $y \in U^{\ominus r_n}$ such that $x \in U_{y}$.  Assume $\gamma$ is parameterized by arc length and let $h_\ddag$ be given by \lemref{diam-finite} and let $D$ denote the intrinsic diameter of $U^{\ominus h_\ddag}$.  Then assuming $r_n \le h_\ddag$, there is a curve $\gamma \subset U^{\ominus r_n}$ of length $L \le D$ joining $y_\star$ and $y$.  
Let $y_0 = y_\star$, $y_j = \gamma(j r_n)$ for $j = 0, \dots, J := \lfloor L/r_n \rfloor$, and then $y_{J+1} = y$.  
We have $\max_{z \in U_{y_j} \cap U_{y_{j+1}}} \|S_{y_j}(z) - S_{y_{j+1}}(z)\| \le 2 C_2 \eps_n$ by the triangle inequality.
We also have $\rho(U_{y_j} \cap U_{y_{j+1}}) \ge r_n$, because $\|y_j - y_{j+1}\| \le r_n$.
Let $v_j$ be such that $B(v_j, r_n/2) \subset U_{y_j} \cap U_{y_{j+1}}$.
Fix $j$ and let $v_{j,0}, \dots, v_{j,d}$ denote a regular simplex inscribed in the ball $B(v_j, r_n/4)$.  Let $\lambda_n \propto r_n$ denote its edge length.   
Then let $x_{j,0}, \dots, x_{s,d} \in \Omega_n$ be such that $\max_k \|x_{j,k} - v_{j,k}\| \le \eps_n$.  When $C_0$ is large enough, $x_{j,0}, \dots, x_{j,d} \in B(v_j, r_n/2)$ by the triangle inequality.  Moreover, $\max_{k,l} \|x_{j,k} - x_{j,l}\| \le \lambda_n + 2 \eps_n$, as well as $\min_{k \ne l} \|x_{j,k} - x_{j,l}\| \ge \lambda_n - 2 \eps_n$.  When $C_0$ is large enough, $F_j := \{x_{j,0}, \dots, x_{j,d}\}$ is therefore an $\eta$-approximate regular simplex, with $\eta \propto \eps_n/r_n$, and minimum edge length $\propto r_n$.
Now, since $\max_{k} \|S_{y_j}(x_{j,k}) - S_{y_{j+1}}(x_{j,k})\| \le 2 C_2 \eps_n$, by \lemref{affine-close-discrete}, for all $z \in \bbR^d$, $\|S_{y_j}(z) - S_{y_{j+1}}(z)\| \le C C_2 \eps_n \|z - x_{j,0}\|/r_n + 2 C_2 \eps_n$ for some $C \propto 1$, assuming $\eps_n/r_n \le 1/C$.  
In particular, by the fact that $\|x - x_{j,0}\| \le \diam(U)$, this gives $\|S_{y_j}(x) - S_{y_{j+1}}(x)\| \le C_3 \eps_n/r_n$ for some $C_3 \propto \diam(U) C_2$.
Hence,
\[
\|S_{y_\star}(x) - S_y(x)\| \le (J+1) C_3 \eps_n/r_n \le C_4 \eps_n/r_n^2,
\]
since $J \le L/r_n \le D/r_n$.  

This being true for any arbitrary $x \in \Omega_n$, we conclude that 
\[\max_{x \in \Omega_n} \|\phi_n(x) - S_{y_\star}(x)\| \le C_4 \eps_n/r_n^2 + C_2 \eps_n \le C_5 \eps_n/r_n^2.\]

\section{Discussion} \label{sec:discussion}

This paper builds on \citep{klein} to provide some theory for ordinal embedding, an important problem in multivariate statistics (aka unsupervised learning).
We leave open two main problems: \\[-.3in]
\bitem
\item What are the optimal rates of convergence for ordinal embedding with all triple and quadruple comparisons? \\[-.25in]
\item What is the minimum size of $K = K_n$ for consistency of ordinal embedding based on the $K$-nearest neighbor distance comparisons?
\eitem

We note that we only studied the large sample behavior of exact embedding methods.  In particular, we did not discuss or proposed any methodology for producing such an embedding.  For this, we refer the reader to \citep{agarwal2007generalized,terada2014local,borg2005modern} and references therein.
In fact, the practice of ordinal embedding raises a number of other questions in terms of theory, for instance: \\[-.23in]
\bitem 
\item How many flawed comparisons can be tolerated?  
\eitem

\subsection*{Acknowledgements}
We are grateful to Vicente Malave for introducing us to the topic and for reading a draft of this paper.  
We also want to thank an associate editor and two anonymous referees for pertinent comments, and for pointing out some typos and errors.  
We learned of the work of Ulrike von Luxburg and her collaborators at the Mathematical Foundations of Learning Theory Workshop held in Barcelona in June 2014.  We are grateful to the organizers, in particular G\'abor Lugosi, for the invitation to participate.
This work was partially supported by the US Office of Naval Research (N00014-13-1-0257).

\bibliographystyle{chicago}
\bibliography{ordinal}

\end{document}

%% file: notation.tex
\newtheorem{thm}{Theorem}
\newtheorem{prp}{Proposition}
\newtheorem{lem}{Lemma}
\newtheorem{cor}{Corollary}

\theoremstyle{remark}

\newtheorem{rem}{Remark}

\def\beq{\begin{equation}} 
\def\eeq{\end{equation}}
\def\beqn{\begin{eqnarray*}}
\def\eeqn{\end{eqnarray*}}
\def\Bitem{\begin{itemize}\setlength{\itemsep}{.2in}}
\def\bitem{\begin{itemize}\setlength{\itemsep}{.05in}}
\def\eitem{\end{itemize}}
\def\Benum{\begin{enumerate}\setlength{\itemsep}{.2in}}
\def\benum{\begin{enumerate}\setlength{\itemsep}{.05in}}
\def\eenum{\end{enumerate}}
\def\bmult{\begin{multline*}}
\def\emult{\end{multline*}}
\def\bcenter{\begin{center}}
\def\ecenter{\end{center}}
\def\bframe{\begin{frame}}
\def\eframe{\end{frame}}
\def\bsplit{\begin{equation*}\begin{split}}
\def\esplit{\end{split}\end{equation*}}

\newcommand{\thmref}[1]{Theorem~\ref{thm:#1}}
\newcommand{\prpref}[1]{Proposition~\ref{prp:#1}}
\newcommand{\lemref}[1]{Lemma~\ref{lem:#1}}
\newcommand{\secref}[1]{Section~\ref{sec:#1}}

\DeclareMathOperator*{\argmin}{arg\, min}
\DeclareMathOperator*{\esssup}{ess\, sup} 
\DeclareMathOperator*{\essinf}{ess\, inf} 
\DeclareMathOperator{\diam}{diam}

\DeclareMathOperator{\tr}{tr}

\DeclareMathOperator{\vol}{Vol}

\def\cC{\mathcal{C}}

\def\cP{\mathcal{P}}

\def\cS{\mathcal{S}}

\def\cX{\mathcal{X}}

\def\bbN{\mathbb{N}}

\def\bbR{\mathbb{R}}

\renewcommand{\P}{\operatorname{\mathbb{P}}}

\def\Bin{\text{Bin}}

\def\eps{\varepsilon}
\def\implies{\ \Rightarrow \ }

\def\comp{\mathsf{c}}
\def\aff{{\rm Aff}}
\def\neigh{{\rm Neigh}}

%% file: ordinal-embedding-v10-post-BEJ.bbl
\begin{thebibliography}{}

\bibitem[\protect\citeauthoryear{Agarwal, Wills, Cayton, Lanckriet, Kriegman,
  and Belongie}{Agarwal et~al.}{2007}]{agarwal2007generalized}
Agarwal, S., J.~Wills, L.~Cayton, G.~Lanckriet, D.~J. Kriegman, and S.~Belongie
  (2007).
\newblock Generalized non-metric multidimensional scaling.
\newblock In {\em International Conference on Artificial Intelligence and
  Statistics}, pp.\  11--18.

\bibitem[\protect\citeauthoryear{Alestalo, Trotsenko, and
  V{\"a}is{\"a}l{\"a}}{Alestalo et~al.}{2001}]{alestalo2001isometric}
Alestalo, P., D.~Trotsenko, and J.~V{\"a}is{\"a}l{\"a} (2001).
\newblock Isometric approximation.
\newblock {\em Israel Journal of Mathematics\/}~{\em 125\/}(1), 61--82.

\bibitem[\protect\citeauthoryear{Aumann and Kruskal}{Aumann and
  Kruskal}{1958}]{aumann1958coefficients}
Aumann, R.~J. and J.~Kruskal (1958).
\newblock The coefficients in an allocation problem.
\newblock {\em Naval Research Logistics Quarterly\/}~{\em 5\/}(2), 111--123.

\bibitem[\protect\citeauthoryear{Borg and Groenen}{Borg and
  Groenen}{2005}]{borg2005modern}
Borg, I. and P.~J. Groenen (2005).
\newblock {\em Modern multidimensional scaling: Theory and applications}.
\newblock Springer.

\bibitem[\protect\citeauthoryear{Burago, Burago, and Ivanov}{Burago
  et~al.}{2001}]{burago2001course}
Burago, D., Y.~Burago, and S.~Ivanov (2001).
\newblock {\em A course in metric geometry}, Volume~33.
\newblock American Mathematical Society Providence.

\bibitem[\protect\citeauthoryear{Cuevas, Fraiman, and Pateiro-L{\'o}pez}{Cuevas
  et~al.}{2012}]{MR2977397}
Cuevas, A., R.~Fraiman, and B.~Pateiro-L{\'o}pez (2012).
\newblock On statistical properties of sets fulfilling rolling-type conditions.
\newblock {\em Adv. in Appl. Probab.\/}~{\em 44\/}(2), 311--329.

\bibitem[\protect\citeauthoryear{Davenport}{Davenport}{2013}]{davenport2013lost}
Davenport, M.~A. (2013).
\newblock Lost without a compass: Nonmetric triangulation and landmark
  multidimensional scaling.
\newblock In {\em Computational Advances in Multi-Sensor Adaptive Processing
  (CAMSAP), 2013 IEEE 5th International Workshop on}, pp.\  13--16. IEEE.

\bibitem[\protect\citeauthoryear{De~Silva and Tenenbaum}{De~Silva and
  Tenenbaum}{2004}]{de2004sparse}
De~Silva, V. and J.~B. Tenenbaum (2004).
\newblock Sparse multidimensional scaling using landmark points.
\newblock Technical report, Technical report, Stanford University.

\bibitem[\protect\citeauthoryear{Ellis, Whitman, Berenzweig, and
  Lawrence}{Ellis et~al.}{2002}]{ellis2002quest}
Ellis, D.~P., B.~Whitman, A.~Berenzweig, and S.~Lawrence (2002).
\newblock The quest for ground truth in musical artist similarity.
\newblock In {\em Proceedings of the International Symposium on Music
  Information Retrieval (ISMIR)}, pp.\  170--177.

\bibitem[\protect\citeauthoryear{Federer}{Federer}{1959}]{MR0110078}
Federer, H. (1959).
\newblock Curvature measures.
\newblock {\em Trans. Amer. Math. Soc.\/}~{\em 93}, 418--491.

\bibitem[\protect\citeauthoryear{Horn and Johnson}{Horn and
  Johnson}{1990}]{MR1084815}
Horn, R.~A. and C.~R. Johnson (1990).
\newblock {\em Matrix analysis}.
\newblock Cambridge University Press, Cambridge.
\newblock Corrected reprint of the 1985 original.

\bibitem[\protect\citeauthoryear{Jamieson and Nowak}{Jamieson and
  Nowak}{2011}]{jamieson2011low}
Jamieson, K.~G. and R.~D. Nowak (2011).
\newblock Low-dimensional embedding using adaptively selected ordinal data.
\newblock In {\em Communication, Control, and Computing (Allerton), 2011 49th
  Annual Allerton Conference on}, pp.\  1077--1084. IEEE.

\bibitem[\protect\citeauthoryear{Kelley}{Kelley}{1975}]{kelley1975general}
Kelley, J.~L. (1975).
\newblock {\em General topology}, Volume~27 of {\em Graduate Texts in
  Mathematics}.
\newblock Springer-Verlag.

\bibitem[\protect\citeauthoryear{Kleindessner and von Luxburg}{Kleindessner and
  von Luxburg}{2014}]{klein}
Kleindessner, M. and U.~von Luxburg (2014).
\newblock Uniqueness of ordinal embedding.
\newblock In {\em Proceedings of The 27th Conference on Learning Theory}, pp.\
  40--67.

\bibitem[\protect\citeauthoryear{Kruskal}{Kruskal}{1964}]{MR0169712}
Kruskal, J.~B. (1964).
\newblock Multidimensional scaling by optimizing goodness of fit to a nonmetric
  hypothesis.
\newblock {\em Psychometrika\/}~{\em 29}, 1--27.

\bibitem[\protect\citeauthoryear{McFee and Lanckriet}{McFee and
  Lanckriet}{2011}]{mcfee2011learning}
McFee, B. and G.~Lanckriet (2011).
\newblock Learning multi-modal similarity.
\newblock {\em The Journal of Machine Learning Research\/}~{\em 12}, 491--523.

\bibitem[\protect\citeauthoryear{Nhat, Vo, Challa, and Lee}{Nhat
  et~al.}{2008}]{nhat2008nonmetric}
Nhat, V. D.~M., N.~Vo, S.~Challa, and S.~Lee (2008).
\newblock Nonmetric mds for sensor localization.
\newblock In {\em 3rd International Symposium on Wireless Pervasive Computing
  (ISWPC)}, pp.\  396--400.

\bibitem[\protect\citeauthoryear{Shepard}{Shepard}{1962a}]{MR0140376}
Shepard, R.~N. (1962a).
\newblock The analysis of proximities: multidimensional scaling with an unknown
  distance function. {I}.
\newblock {\em Psychometrika\/}~{\em 27}, 125--140.

\bibitem[\protect\citeauthoryear{Shepard}{Shepard}{1962b}]{MR0173342}
Shepard, R.~N. (1962b).
\newblock The analysis of proximities: multidimensional scaling with an unknown
  distance function. {II}.
\newblock {\em Psychometrika\/}~{\em 27}, 219--246.

\bibitem[\protect\citeauthoryear{Shepard}{Shepard}{1966}]{shepard1966metric}
Shepard, R.~N. (1966).
\newblock Metric structures in ordinal data.
\newblock {\em Journal of Mathematical Psychology\/}~{\em 3\/}(2), 287--315.

\bibitem[\protect\citeauthoryear{Sikorska and Szostok}{Sikorska and
  Szostok}{2004}]{sikorska2004mappings}
Sikorska, J. and T.~Szostok (2004).
\newblock On mappings preserving equilateral triangles.
\newblock {\em Journal of Geometry\/}~{\em 80\/}(1-2), 209--218.

\bibitem[\protect\citeauthoryear{Suppes and Winet}{Suppes and
  Winet}{1955}]{suppes1955axiomatization}
Suppes, P. and M.~Winet (1955).
\newblock An axiomatization of utility based on the notion of utility
  differences.
\newblock {\em Management Science\/}, 259--270.

\bibitem[\protect\citeauthoryear{Terada and Von~Luxburg}{Terada and
  Von~Luxburg}{2014}]{terada2014local}
Terada, Y. and U.~Von~Luxburg (2014).
\newblock Local ordinal embedding.
\newblock In {\em Proceedings of the 31st International Conference on Machine
  Learning (ICML-14)}, pp.\  847--855.

\bibitem[\protect\citeauthoryear{Vestfrid}{Vestfrid}{2003}]{vestfrid2003linear}
Vestfrid, I.~A. (2003).
\newblock Linear approximation of approximately linear functions.
\newblock {\em aequationes mathematicae\/}~{\em 66\/}(1-2), 37--77.

\bibitem[\protect\citeauthoryear{Von~Luxburg and Alamgir}{Von~Luxburg and
  Alamgir}{2013}]{von2013density}
Von~Luxburg, U. and M.~Alamgir (2013).
\newblock Density estimation from unweighted k-nearest neighbor graphs: a
  roadmap.
\newblock In {\em Advances in Neural Information Processing Systems}, pp.\
  225--233.

\bibitem[\protect\citeauthoryear{Waldmann}{Waldmann}{2014}]{waldmann2014topology}
Waldmann, S. (2014).
\newblock {\em Topology: An Introduction}.
\newblock Springer International Publishing.

\bibitem[\protect\citeauthoryear{Young and Hamer}{Young and
  Hamer}{1987}]{young1987multidimensional}
Young, F.~W. and R.~M.~E. Hamer (1987).
\newblock {\em Multidimensional scaling: History, theory, and applications.}
\newblock Lawrence Erlbaum Associates, Inc.

\end{thebibliography}
